\documentclass{amsart}

\usepackage[utf8]{inputenc}

\usepackage{amssymb}
\usepackage[hidelinks]{hyperref}
\usepackage[textsize=small]{todonotes}
\usepackage{tikz}
\usetikzlibrary{decorations.text}
\usepackage{xcolor}
\usepackage[normalem]{ulem}
\usepackage{mathrsfs}
\usepackage{graphicx} % Required for inserting images

\theoremstyle{plain}
\newtheorem{theorem}{Theorem}[section]
\newtheorem*{theorem-wo-number}{Theorem}
\newtheorem{lemma}[theorem]{Lemma}
\newtheorem{corollary}[theorem]{Corollary}
\newtheorem{proposition}[theorem]{Proposition}

\theoremstyle{definition}
\newtheorem{definition}[theorem]{Definition}

\newtheorem{example}[theorem]{Example}

\theoremstyle{remark}
\newcommand{\bQ}{\mathbb{Q}}
\newcommand{\Q}{\bQ}
\newcommand{\bN}{\mathbb{N}}
\newcommand{\N}{\bN}

\newcommand{\cA}{\mathcal{A}}
\newcommand{\cB}{\mathcal{B}}
\newcommand{\cC}{\mathcal{C}}

\newcommand{\cI}{\mathcal{I}}
\newcommand{\I}{\cI}
\newcommand{\cJ}{\mathcal{J}}
\newcommand{\J}{\cJ}
\newcommand{\cK}{\mathcal{K}}

\newcommand{\cP}{\mathcal{P}}

\newcommand{\bnumber}{\mathfrak{b}}

\newcommand{\dnumber}{\mathfrak{d}}

\DeclareMathOperator{\add}{add}
\DeclareMathOperator{\non}{non}
\DeclareMathOperator{\cov}{cov}
\DeclareMathOperator{\cof}{cof}
\DeclareMathOperator{\adds}{\add^*}
\DeclareMathOperator{\nons}{\non^*}
\DeclareMathOperator{\covs}{\cov^*}
\DeclareMathOperator{\cofs}{\cof^*}
\newcommand{\fin}{\mathrm{Fin}}
\newcommand{\Fin}{\fin}
\newcommand{\conv}{\mathrm{conv}}

\newcommand{\BI}{\mathrm{BI}} 
\DeclareMathOperator{\closure}{cl}

\DeclareMathOperator{\FinBW}{FinBW}
\DeclareMathOperator{\ot}{ot}

\renewcommand{\subset}{\subseteq}

\begin{document}

%%%%%%%%%%%%%%%%%%%%%%%%%%%%%%%%%%%%%%%%%%%%%%%%%%
%%%%% TITLE
%%%%%%%%%%%%%%%%%%%%%%%%%%%%%%%%%%%%%%%%%%%%%%%%%%

\title{Critical ideals for compact spaces}

%%%%%%%%%%%%%%%%%%%%%%%%%%%%%%%%%%%%%%%%%%%%%%%%%%
%%%%% AUTHORS
%%%%%%%%%%%%%%%%%%%%%%%%%%%%%%%%%%%%%%%%%%%%%%%%%%

\author[R.~Filip\'{o}w]{Rafa\l{} Filip\'{o}w}
\address[Rafa\l{}~Filip\'{o}w]{Institute of Mathematics\\ Faculty of Mathematics, Physics and Informatics\\ University of Gda\'{n}sk\\ ul.~Wita Stwosza 57\\ 80-308 Gda\'{n}sk\\ Poland}
\email{rafal.filipow@ug.edu.pl}
\urladdr{\url{http://mat.ug.edu.pl/~rfilipow}}

\author[M.~Kowalczuk]{Ma\l{}gorzata Kowalczuk}
\address[Ma\l{}gorzata Kowalczuk]{Institute of Mathematics\\ Faculty of Mathematics, Physics and Informatics\\ University of Gda\'{n}sk\\ ul.~Wita Stwosza 57\\ 80-308 Gda\'{n}sk\\ Poland}
\email{m.kowalczuk.090@studms.ug.edu.pl}

\author[A.~Kwela]{Adam Kwela}
\address[Adam Kwela]{Institute of Mathematics\\ Faculty of Mathematics\\ Physics and Informatics\\ University of Gda\'{n}sk\\ ul.~Wita  Stwosza 57\\ 80-308 Gda\'{n}sk\\ Poland}
\email{Adam.Kwela@ug.edu.pl}
\urladdr{\url{https://mat.ug.edu.pl/~akwela}}

\thanks{The third-listed author was supported by the Polish National Science Centre project OPUS No. 2024/53/B/ST1/02494.}

%%%%%%%%%%%%%%%%%%%%%%%%%%%%%%%%%%%%%%%%%%%%%%%%%%
%%%%% DATE
%%%%%%%%%%%%%%%%%%%%%%%%%%%%%%%%%%%%%%%%%%%%%%%%%%

\date{\today}

%%%%%%%%%%%%%%%%%%%%%%%%%%%%%%%%%%%%%%%%%%%%%%%%%%
%%%%% MSC (MATHEMATICAL SUBJECT CLASSIFICATION)
%%%%%%%%%%%%%%%%%%%%%%%%%%%%%%%%%%%%%%%%%%%%%%%%%%


% 03E05 (1980–now) Other combinatorial set theory
% 03E10 (1980–now) Ordinal and cardinal numbers
% 03E15 (1980–now) Descriptive set theory [See also 28A05, 54H05]
% 03E17 (2000–now) Cardinal characteristics of the continuum
% 03E75 (1980-now) Applications of set theory 
%
% 26A03 (1973–now) Foundations: limits and generalizations, elementary topology of the line
%
% 40A05 (1973-now) Convergence and divergence of series and sequences 
% 40A35 (2010–now) Ideal and statistical convergence [See also 40G15]
%
% 54A20 Convergence in general topology (sequences, filters, limits, convergence spaces, etc.)
% 54H05 (1973–now) Descriptive set theory (topological aspects of Borel, analytic, projective, etc. sets) [See also 03E15, 26A21, 28A05]

%%%%%%%%%%%%%%%%%%%%%%%%%%%%%%%%%%%%%%%%%%%%%%%%%%
%%%%% KEYWORDS
%%%%%%%%%%%%%%%%%%%%%%%%%%%%%%%%%%%%%%%%%%%%%%%%%%


%%%%%%%%%%%%%%%%%%%%%%%%%%%%%%%%%%%%%%%%%%%%%%%%%%
%%%%% ABSTRACT
%%%%%%%%%%%%%%%%%%%%%%%%%%%%%%%%%%%%%%%%%%%%%%%%%%

\begin{abstract}
For each countable ordinal $\alpha$, we introduce an ideal $\conv_\alpha$ and use it to characterize the class of all compact countable spaces which are homeomorphic to the space $\omega^{\alpha}\cdot$ $n+1$ with the order topology. The characterization is expressed in terms of finding a convergent subsequence  defined on a set  not belonging to $\conv_\alpha$.
\end{abstract}

%%%%%%%%%%%%%%%%%%%%%%%%%%%%%%%%%%%%%%%%%%%%%%%%%%
%%%%% MAKE TITLE
%%%%%%%%%%%%%%%%%%%%%%%%%%%%%%%%%%%%%%%%%%%%%%%%%%

\maketitle

%%%%%%%%%%%%%%%%%%%%%%%%%%%%%%%%%%%%%%%%%%%%%%%%%%
%%%%% Table of contents
%%%%%%%%%%%%%%%%%%%%%%%%%%%%%%%%%%%%%%%%%%%%%%%%%%

% \setcounter{tocdepth}{2}
% \tableofcontents

%%%%%%%%%%%%%%%%%%%%%%%%%%%%%%%%%%%%%%%%%%%%%%%%%%
%%%%% SECTION
%%%%%%%%%%%%%%%%%%%%%%%%%%%%%%%%%%%%%%%%%%%%%%%%%%

\section{Introduction}\label{sec:intro}

See Sections~\ref{sec:prelim} and \ref{sec:critical-ideals} for notions and notations used in the introduction.

For an ideal $\I$ on $\omega$, we define the class 
$\FinBW(\I)$ of all topological  spaces $X$ having the property that for every sequence $(x_n)_{n\in \omega}$ in $X$ there exists $A\notin \I$ such that the subsequence $(x_n)_{n\in A}$ is convergent in $X$ (\cite[Definition~1.1]{MR4584767}).

We  see that  $\FinBW(\Fin)$ coincides with the class of all sequentially compact  spaces, and one can show  that $\FinBW(\Fin\otimes\Fin)$ coincides with the class of all finite spaces, whereas $\FinBW(\BI)$ coincides with the class of all boring spaces (\cite[Theorem~6.5]{MR4584767}).
Using the Kat\v{e}tov order one can show 
that the above mentioned ideals are critical, in a sense, for considered classes of spaces.

\begin{theorem}[{\cite[Theorem~6.5]{MR4584767}}] 
\label{thm:critical-for-finite-boring-all}
Let $\I$ be an ideal on $\omega$.
\begin{enumerate}
        \item $\I\leq_K\Fin$ $\iff$ $\FinBW(\I)$ coincides with the class of all sequentially compact spaces.
        \item $\Fin\otimes \Fin \leq_K\I$ $\iff$ $\FinBW(\I)$ coincides with the class of all finite spaces.
        \item The following conditions are equivalent.
        \begin{enumerate}
            \item $\BI  \leq_K\I$ and $\Fin\otimes\Fin\not\leq_K \I$.
            \item $\FinBW(\I)$ coincides with the class of all boring  spaces.
            \end{enumerate}
    \end{enumerate}
\end{theorem}

In \cite[Section~2.7]{alcantara-phd-thesis}, the author proved  that 
$$\conv\not\leq_K \I \iff [0,1]\in \FinBW(\I).$$
Using the above  equivalence, one can  show that the ideal $\conv$ turns out to be critical for uncountable compact metric spaces as shown in the following  theorem.

\begin{theorem}
\label{thm:Meza-compact-metric-versus-conv}
Let $\I$ be an ideal on $\omega$.
The following conditions are equivalent.  
\begin{enumerate}
    \item 
$\conv\not\leq_K \I$.

\item $\FinBW(\I)$ contains an uncountable compact metric space.

\item $\FinBW(\I)$  contains  all compact metric spaces.

\end{enumerate}
\end{theorem}

In the realm of  compact countable spaces, we have the following theorem 
of Mazurkiewicz and Sierpi\'{n}ski   
which characterizes these spaces in terms of countable ordinals with the order topology.
Note that every  compact countable space is metrizable (see e.g.~\cite[Proposition~8.5.7]{MR296671}), and consequently it is  sequentially compact.

\begin{theorem}[{\cite{Mazurkiewicz-Sierpinski}, see also \cite[Theorem~8.6.10]{MR296671}}]
\label{thm:Mazurkiewicz-Sierpinski}
A countable space $X$ is compact  if and only if 
$X$ is homeomorphic to the space $\omega^\alpha\cdot n+1$ with the order topology for some countable ordinal $\alpha$ and $n\in \omega$.
\end{theorem}

Let $\mathbb{K}$ denote  the class of all countable  compact spaces
and 
 $\mathbb{K}_{\alpha}$
denote  the class of all  compact spaces which are homeomorphic to  
$\omega^{\beta}\cdot n+1$ for some $\beta\leq \alpha$ and $n\in \omega$.  
By Mazurkiewicz-Sierpi\'{n}ski theorem we obtain the equality:
$$\mathbb{K} = \bigcup_{\alpha<\omega_1}\mathbb{K}_\alpha.$$
The main objective of this paper is to define ideals $\conv_\alpha$ and prove the following theorem which shows that $\conv_\alpha$ is a critical ideal for $\mathbb{K}_\alpha$.

\begin{theorem-wo-number}
Let  $\I$ be an ideal on  $\omega$ and  $\alpha$ be a countable ordinal.
The following conditions are equivalent.
\begin{enumerate}
    \item 
$\conv_{(1+\alpha)+1} \leq_K \I$ and $\conv_{1+\alpha} \not\leq_K \I$.
\item 
$\FinBW(\I)\cap \mathbb{K}=\mathbb{K}_{\alpha}.$
\end{enumerate}
\end{theorem-wo-number}

The next  theorem (interesting in its own right) is a crucial ingredient in the proof of the above theorem.

\begin{theorem-wo-number}
     For any ideal $\I$ on $\omega$ and countable ordinal $\alpha$,   
    $$\conv_{1+\alpha}\not\leq_K \I \iff \omega^{\alpha}+1 \in \FinBW(\I).$$ 
\end{theorem-wo-number}

The critical ideals $\conv_\alpha$ are defined in Section~\ref{sec:critical-ideals} where we also prove some basic properties of  these ideals. The proofs of the above theorems are given in Section~\ref{sec:proofs}.
In Section~\ref{sec:Kat-among-cons}, we prove that the sequence $(\conv_\alpha)_{\alpha<\omega_1}$ of the critical ideals is strictly decreasing in the Kat\v{e}tov order.
In Section~\ref{sec:conv-alpha-vs-conv}, we show that the sequence  $(\conv_\alpha)_{\alpha<\omega_1}$ is critical for the space $\omega_1$, but there is no single ideal which is critical for the space $\omega_1$.
In Section~\ref{sec:properties-of-critical-ideals},  we present some additional properties of the critical ideals concerning the property Kat, the property $P^-$, Borel complexity and cardinal characteristics.

%%%%%%%%%%%%%%%%%%%%%%%%%%%%%%%%%%%%%%%%
%%%
%%%%%%%%%%%%%%%%%%%%%%%%%%%%%%%%%%%%%%%%

\section{Preliminaries}
\label{sec:prelim}

All topological spaces considered in the paper are assumed to be Hausdorff.
A sequentially compact space $X$ is called \emph{boring} (\cite[Definition~3.21]{MR4584767}) if there exists a finite set $F\subseteq X$ such that each one-to-one convergent sequence in $X$ converges to some point from  $F$.
 
Recall that an ordinal number $\alpha$ is equal to the set of all ordinal numbers less than $\alpha$. 
In particular, the  smallest infinite ordinal number $\omega=\{0,1,\dots\}$ is equal to the set of all natural numbers $\N$, and each natural number $n = \{0,\dots,n-1\}$ is equal  to the set of all natural numbers less than $n$.
Using this identification, we can for instance write $n\in k$ instead of $n<k$ and $n<\omega$ instead of $n\in \omega$ or $A\cap n$ instead of $A\cap \{0,1,\dots,n-1\}$. 
Moreover, for ordinals $\alpha<\beta$, we write 
$[\alpha,\beta]$ to denote the set of all ordinals $\xi$ such that  $\alpha\leq \xi\leq \beta$, and similarly for $(\alpha,\beta)$ and other intervals.
If $A$ is a set of ordinals, we write $\ot(A)$ to denote the \emph{order type} of $A$
 i.e.~the unique ordinal $\alpha$ such that $A$ and $\alpha$ are isomorphic.

An \emph{ideal} on a set $X$ is a nonempty family $\I\subseteq\cP(X)$ which is closed under taking finite unions (i.e.~if $A,B\in \I$ then $A\cup B\in\I$)
and subsets (i.e.~if $A\subseteq B$ and $B\in\I$ then $A\in\I$).
For an ideal $\I$,  we write $\I^+=\{A\subseteq X: A\notin\I\}$ and call it the \emph{coideal of $\I$}, 
and we write $\I^*=\{A\subseteq X: X\setminus A\in\I\}$ and call it the \emph{dual filter of $\I$}.
An ideal $\I$ is \emph{tall} (a.k.a. \emph{dense}) if for every infinite $A\subseteq X$ there is an infinite $B\in\I$ such that $B\subseteq A$. 

By identifying sets of natural numbers with their characteristic functions,
we equip $\cP(\omega)$ with the topology of the Cantor space $\{0,1\}^\omega$ and therefore
we can assign topological complexity to ideals on $\omega$.
In particular, an ideal $\I$ is Borel ($F_\sigma$, analytic, resp.) if $\I$ is Borel ($F_\sigma$, analytic, resp.) as a subset of the Cantor space.

\begin{example}\  
\begin{enumerate}

    \item 
     $\fin(X)$ is  the ideal of all finite subsets of $X$.
    We  write $\fin$ instead of $\fin(\omega)$.
    
\item If $\phi:\mathcal{P}(\omega)\to[0,\infty]$ is a countably additive  measure  such that $\phi(\omega)=\infty$, then
$\I_\phi$ given by:
$$A\in\I_\phi\iff \phi(A)<\infty$$
is a \emph{summable ideal}. $\fin$ is a summable ideal given by the measure $\phi(A)=|A|$ for all $A\subseteq\omega$. Each summable ideal is $\Sigma^0_2$.

\item $\I_d$ is the ideal of all subsets of $\omega$  of asymptotic density zero:
$$A\in \I_d\iff \limsup_{n\to\infty}\frac{|A\cap n|}{n}=0.$$
It is a $\Pi^0_3$ ideal.

\item $\{\emptyset\} \otimes \Fin $ is the ideal on $\omega \times\omega$ defined by $$A\in \{\emptyset\}\otimes \fin \iff \forall i \, (|\{j: (i,j)\in A\}|<\omega). $$
It is a $\Pi^0_3$ ideal.

\item 
$\Fin^2 = \Fin\otimes\Fin$ is the ideal on $\omega\times \omega$ defined by
$$A\in \fin\otimes \fin \iff \exists i_0\, \forall i\geq i_0\, (|\{j: (i,j)\in A\}|<\omega).$$
 It is a $\Sigma^0_4$ ideal.
 
\item $\BI$ is the ideal  on $\omega\times \omega \times\omega$ 
introduced in \cite[Definition~4.1]{MR4584767}
and defined by 
\begin{equation*}
    \begin{split}
A\in \BI 
\iff 
\exists i_0\, \left[\right.
&
\,\, \forall i< i_0\, (\, \{(j,k) : (i,j,k)\in A\}\in \Fin\otimes\Fin)
\  \land 
\\&
\left.
\forall i\geq i_0\, (|\{(j,k) : (i,j,k)\in A\}|<\omega)
\right].
    \end{split}
\end{equation*}
It is a $\Sigma^0_4$ ideal. 
\end{enumerate}\end{example}

The vertical section of a set   $A\subseteq X\times Y$ at a point $x\in X$ is defined by $A_{(x)} = \{y\in Y : (x,y)\in A\}$. 
For ideals $\I$ and $\J$ on $X$ and $Y$ respectively, we define the following  ideal (called the \emph{Fubini product} of $\I$ and $\J$):
$$\I\otimes \J = \{A\subseteq X\times Y: \{x\in X:A_{(x)}\notin\J\}\in \I\}.$$
Using the notion of the Fubini product, one can see that  (\cite{MR4584767}) 
$$\BI=(\{\emptyset\}\otimes\fin^2)\cap(\Fin\otimes \Fin(\omega^2)).$$ 

An ideal $\I$ is a \emph{P-ideal} if  for every sequence $(A_n)_{n\in\omega}$ of elements of $\I$ there is $A\in\I$ such that $A_n\setminus A$ is finite for all $n\in\omega$. The ideals $\fin$, $\I_d$, $\{\emptyset\} \otimes \Fin $ and all summable ideals are P-ideals,  whereas $\Fin^2$ and $\BI$ are not P-ideals.

Let $\I$ and $\J$ be ideals on $X$ and $Y$ respectively.
We say that $\I$ and $\J$ are \emph{isomorphic} (in short $\I\approx\J$) 
if there exists a bijection $f:X\rightarrow Y$ such that $A\in\I\iff f[A]\in\J$ for every $A\subset X$.
We say that \emph{$\J$ is below $\I$ in  the Kat\v{e}tov order} (in short $\J\leq_{K}\I$) if there is a function $f:X\rightarrow Y$ such that $f^{-1}[A]\in\I$ for every $A\in \I$. 
We say that ideals $\I$ and $\J$ are \emph{$\leq_K$-equivalent} if $\I\leq_K\J$ and $\J\leq_K\I$.

A relationship between Kat\v{e}tov order and classes $\FinBW(\I)$ is expressed in the following theorem which will be used in other sections.

\begin{theorem}[{\cite[Corollary~10.2]{MR4584767}}]
\label{thm:KAT-vs-FinBW} 
If $\I\leq_K\J$, then $\FinBW(\J)\subseteq \FinBW(\I)$.
\end{theorem}

%%%%%%%%%%%%%%%%%%%%%%%%%%%%%%%%%%%%%%%%%%%%%%%%%%
%%%%% SECTION
%%%%%%%%%%%%%%%%%%%%%%%%%%%%%%%%%%%%%%%%%%%%%%%%%%

\section{Critical ideals}
\label{sec:critical-ideals}

For a  topological space $X$, 
we write $c(X)$ to denote the set of all convergent sequences in $X$.
For a subset $D\subseteq X$ of a topological space $X$, we write  $\conv(D)$ to denote 
the ideal on  $D$ consisting of all subsets of $D$ which can be covered by ranges of finitely many sequences in $D$ which are convergent in $X$ i.e.
$A\in \conv(D)$
if and only if $A\subseteq D$ and there exist $k\in \omega$ and sequences $(d^{(i)}_n)_{n\in \omega}\in c(X)\cap D^\omega$ for $i<k$ such that 
$$A \subseteq \bigcup_{i<k}\left\{d^{(i)}_n:n\in \omega\right\}.$$

A point $p\in X$ is an \emph{accumulation} (a.k.a.~\emph{limit}) \emph{point} of a set $A\subseteq X$ in a topological space $X$ if $p\in \closure(A\setminus\{p\})$.
By $A^d$ we denote the \emph{derived set of $A$} i.e.~the set of  all accumulation  points of $A$.

\begin{lemma}
\label{lem:conv-ideal-via-derivative-set}
    Let $X$ be a sequentially compact space.
    Let $D\subseteq X$ be a countable infinite subset of $X$.
\begin{enumerate}
    \item \label{lem:conv-ideal-via-derivative-set:item-1}  For every $A\subseteq D$, 
$$A\in \conv(D) \iff A^{d} \text{ is finite}.$$

\item For every infinite set $A\subseteq D$ there exists an infinite set $B\subseteq A$ such that  $B\in \conv(D)$. In particular, the ideal  $\conv(D)$ is tall.\label{lem:conv-ideal-via-derivative-set:item-4}

\end{enumerate}
\end{lemma}

\begin{proof}

(\ref{lem:conv-ideal-via-derivative-set:item-1}, $\implies$) Take any $A\in \conv(D)$. There exists $k\in \omega$ and sequences $(d^{(i)}_n)_{n\in \omega}\in c(X)\cap D^\omega$ for $i<k$ such that $A \subseteq \bigcup_{i<k} \{d^{(i)}_n:n\in \omega \}.$ The sequence $(d^{(i)}_n)_{n\in \omega}$ is convergent for each $i$, so let $\{d_i : i<k\}$ be the set of their limits. Then $A^d \subseteq (\bigcup_{i<k} \{d^{(i)}_n:n\in \omega \})^d\subseteq \{d_i : i<k\}$ which is finite.

(\ref{lem:conv-ideal-via-derivative-set:item-1}, $\impliedby$) Assume that $A$ is infinite, $A^d$ is finite and let $k$ be the cardinality of the set.
We will prove the statement by induction on $k\in \omega$. 

For $k=1$, let $A^d=\{d\}$. Since $A$ is countable infinite, we can enumerate its elements injectively as $\{a_n : n \in \omega\}$. We claim that the sequence $(a_n)_{n\in \omega}$ is convergent to $d$. Let $U$ be an open neighborhood of $d$ and suppose for the sake of contradiction that there exists infinitely many elements of the set $\{a_n : n \in \omega\}$ that are not in $U$. Since $X$ is sequentially compact, there exists an infinite set $C\subseteq \omega$ such that $(a_n)_{n\in C}$ is convergent to a point $p$ that is different from $d$, so $p\in A^d$ which is a contradiction with our initial assumption. Consequently we can cover $A$ with a convergent sequence, so $A\in \conv(D)$.

Suppose the statement is true for $k$. Take $A^d=\{d_i : i<k+1\}$. Since $X$ is Hausdorff we can find a collection of open sets $\{U_i\subset X : d_i\in U_i \}$ that are pairwaise disjoint. Define a partition of $A$ by $B_1=A\cap U_k$ and $ B_2=A\setminus U_k$. 

We claim $B_1^d\subseteq \{d_k\}$ and $B_2^d\subseteq \{d_i : i<k \}$. 
Take $i<k$. Then $B_1\cap U_i \subseteq U_k \cap U_i= \emptyset$, so $d_i\notin B_1^d$. 
Similarly,  $B_2 \cap U_k = \emptyset $, so $d_k\notin B_2^d$.

So by the induction hypothesis $B_1 \in \conv(D)$ and $B_2 \in \conv(D)$ and therefore $A\in\conv(D)$

(\ref{lem:conv-ideal-via-derivative-set:item-4})
Take an infinite set $A\subseteq D$. 
Take distinct points $a_n \in A$ for $n\in \omega$.
Since $X$ is sequentially compact, there is a convergent subsequence $(a_{k_n})_{n\in \omega}$. Then $B=\{a_{k_n}:n\in \omega\}$ is an infinite subset of $A$ such that $B\in \conv(D)$.
\end{proof}

\begin{definition}\ 
\begin{enumerate}
    \item 
If $X=[0,1]$ with the Euclidean topology, we define the ideal
$$\conv = \conv(\Q\cap [0,1]).$$
\item For a countable  ordinal $\alpha\geq 2$ and 
the space $X=\omega^{\alpha}+1$ with the order topology, we define the ideal
$$\conv_{\alpha} = \conv(\omega^{\alpha}+1).$$
In the rest of the paper, we write ``$\alpha$'' to mean ''a countable ordinal  $\alpha\geq 2$''.
\end{enumerate}
\end{definition}

Below we show that the Borel complexity of  $\conv_\alpha$ is at most $\Sigma^0_4$, and in Corollary~\ref{cor:conv-alpha-is-not-Pi-0-4} we will show that this complexity is not lower. 

\begin{proposition}
\label{prop:Borel_complexity}
$\conv_\alpha\in \Sigma^0_4$ for every $\alpha$.
\end{proposition}

\begin{proof}
We use a standard quantifier-counting argument. Let $\{U_n:n\in \omega\}$ be a countable basis for the topology of $X=\omega^\alpha+1$. 
By Lemma~\ref{lem:conv-ideal-via-derivative-set}, $A\in \conv_\alpha\iff |A^d|<\omega$, hence  
$$A\in \conv_\alpha \iff \exists F\in [X]^{<\omega}\, \forall \beta\in X\setminus F\, \exists n\in \omega\,\left(\beta\in U_n \land \forall \gamma\in U_n (\gamma\notin A)\right).$$
Since the sets $\{A\subseteq X: \gamma\notin A\}$ are  closed in $X$ for every $\gamma$, we can count  quantifiers to  obtain that $\conv_\alpha$ is $F_{\sigma\delta\sigma}$.
\end{proof}

The following \emph{purely combinatorial} characterization of members of the topologically defined ideal $\conv_\alpha$ will be used  repeatedly in the next sections.

\begin{proposition}
\label{prop:gwiazdka}
Let  $A\subseteq \omega^{\alpha}+1$.
The following conditions are equivalent.
\begin{enumerate}
    \item $A\in \conv_\alpha$.
    \item For every increasing sequence $(\lambda_n)_{n<\omega}$ in $\omega^\alpha$, the set
    $A\cap [\lambda_n, \lambda_{n+1}]$
is finite for all but finitely many $n$.
\end{enumerate}\end{proposition}

\begin{proof}
$(1)\implies (2)$
Take $A\in \conv_\alpha$ and suppose for the sake of contradiction that there exists 
an increasing sequence $(\lambda_n)_{n\in \omega}$ in $\omega^\alpha$ and an infinite set $B\subseteq\omega$ such that 
the set $A\cap [\lambda_n , \lambda_{n+1}]$
is infinite for every $n \in B$.
Since $\omega^\alpha+1$ is sequentially compact, we can find an accumulation points of $A$ in $(\lambda_{n} , \lambda_{n+1}]$ for every $n\in B$. 
Then $A^d$ is infinite, so $A\notin \conv_\alpha$ by Lemma~\ref{lem:conv-ideal-via-derivative-set},  a contradiction.

$(2)\implies (1)$
Take $A\notin \conv_\alpha$. 
By Lemma~\ref{lem:conv-ideal-via-derivative-set}, $A^d$ is infinite, so we can pick a strictly increasing sequence $(\lambda_n)_{n\in \omega}$ in $A^d$.
Since successor ordinals are isolated points,  $\lambda_n$ is a limit ordinal for each $n$.
Then $(\xi,\lambda_{n+1}]$ is a neighborhood of $\lambda_{n+1}$ for every $\xi\in (\lambda_n,\lambda_{n+1})$, 
and consequently
$A\cap (\lambda_n,\lambda_{n+1}]$ is infinite for every $n\in \omega$.
\end{proof}

 In the case of successor ordinals, we can obtain the following extension of the above characterization. 

\begin{proposition}
    Let  $A\subseteq \omega^{\alpha+1}+1$.
The following conditions are equivalent.
\begin{enumerate}
    \item $A\in \conv_{\alpha+1}$.
    \item For every increasing sequence $(\lambda_n)_{n<\omega}$ in $\omega^{\alpha+1}$, the set
    $A\cap [\lambda_n, \lambda_{n+1}]$
is finite for all but finitely many $n$.

\item 
\begin{enumerate}
\item  $A\cap (\omega^\alpha\cdot n, \omega^\alpha\cdot (n+1)]$ is finite for all but finitely many $n$, and 
\item  $A\cap (\omega^\alpha\cdot n, \omega^\alpha\cdot (n+1)] \in \conv((\omega^\alpha\cdot n, \omega^\alpha\cdot (n+1)])$  for all  $n$.
\end{enumerate}

\end{enumerate}
\end{proposition}

\begin{proof}
$(1)\iff(2)$ This is proved in Proposition~\ref{prop:gwiazdka}.

$(2)\implies(2a)$ Take $\lambda_n=\omega^\alpha\cdot n$ and apply item (2).

$(2)\implies(2b)$ Let  $n\in \omega$. Then we obtain item (3b) by applying  item (2) to every increasing sequence $(\lambda_k)_{k\in \omega}$ in  $(\omega^\alpha\cdot n, \omega^\alpha\cdot (n+1)]$.

$(3)\implies(1)$
Let $n_0$ be such that 
$A\cap (\omega^\alpha\cdot n, \omega^\alpha\cdot (n+1)]$ is finite for every $n\geq n_0$.
Then
$$A^d \subseteq \bigcup_{n<n_0} (A\cap (\omega^\alpha\cdot n, \omega^\alpha\cdot (n+1)])^d \cup \{\omega^{\alpha+1}\}.$$
Since  
$A\cap (\omega^\alpha\cdot n, \omega^\alpha\cdot (n+1)] \in \conv((\omega^\alpha\cdot n, \omega^\alpha\cdot (n+1)])$
for every $n<n_0$,
we obtain that $A^d$ 
is finite, hence $A\in \conv_{\alpha+1}$.
\end{proof}

Since $\omega^\alpha+1$ and $(\omega^\alpha\cdot n,\omega^\alpha\cdot (n+1)]$ are order isomorphic, we obtain that the ideals $\conv_\alpha$ and $\conv((\omega^\alpha\cdot n,\omega^\alpha\cdot (n+1)])$ are isomorphic for every $n$.
Then we can use  item (3) of the above proposition to  obtain an ideal isomorphic to $\conv_{\alpha+1}$.

\begin{corollary}
\label{cor:conv-alpha-as-products}
The ideal $\conv_{\alpha+1}$ is isomorphic to the ideal 
$$\left(\Fin\otimes \Fin(\omega^\alpha+1)\right)\cap \left( \{\emptyset\}\otimes \conv_\alpha \right).$$
\end{corollary}

We finish this section with a proposition which  shows that the results from Section~\ref{sec:proofs} extend Theorem~\ref{thm:critical-for-finite-boring-all}.

\begin{proposition}
\label{prop:conv-for-1-2-3}
The ideals $\conv_2$ and $\fin\otimes \fin$ ($\conv_3$ and $\BI$, resp.)
are isomorphic. 
\end{proposition}

\begin{proof}
Apply Corollary~\ref{cor:conv-alpha-as-products}
and the fact that 
$\BI=(\{\emptyset\}\otimes\fin^2)\cap(\Fin\otimes \Fin(\omega^2))$.
\end{proof}

%%%%%%%%%%%%%%%%%%%%%%%%%%%%%%%%%%%%%%%%%%%%%%%%%%
%%%%% SECTION
%%%%%%%%%%%%%%%%%%%%%%%%%%%%%%%%%%%%%%%%%%%%%%%%%%

\section{Proofs of the main theorems}
\label{sec:proofs}

\begin{lemma}
\label{lem:P-like-property}
Let  $\I$ be an ideal on $\omega$.
If $f:\omega \to \omega^{\alpha}$ is a function such that $f^{-1}(\xi)\in \I$ for each $\xi\in \omega^\alpha$, then for any set $B\subseteq \omega$ such that $ B\notin \I$
    there exists an increasing sequence $(\lambda_n)_{n\in \omega}$ in $\omega^{\alpha}$ such that 
    \begin{enumerate}
        \item $B\cap f^{-1}[[\lambda_n, \lambda_{n+1})]\in \I$ for each $n$, and 
        \item $B\cap f^{-1}[[\lambda_0,\lambda)]\notin \I$, where $\lambda = \sup_n \lambda_n$.
    \end{enumerate}
    \end{lemma}
\begin{proof}
We prove the lemma by induction on $\alpha$.
For $\alpha=1$, the sequence $\lambda_n=n$ works.
Assume the lemma holds for any $\beta<\alpha$.

\emph{Successor case:} $\alpha$ is a successor ordinal, so $\alpha=\delta+1$.

If $B\cap f^{-1}[[\omega^{\delta}\cdot n,\omega^{\delta}\cdot (n+1))]\in \I$ for each $n<\omega$, the sequence $\lambda_n = \omega^{\delta}\cdot n$ works.
Now, assume there is $n<\omega$ with $C = B\cap f^{-1}[[\omega^{\delta}\cdot n,\omega^{\delta}\cdot (n+1))]\notin \I$.
Since $[\omega^{\delta}\cdot n,\omega^{\delta}\cdot (n+1))$ is order isomorphic to $\omega^\delta$, we can apply the inductive hypothesis to obtain an increasing sequence 
 $(\lambda_n)_{n\in \omega}$ in $[\omega^{\delta}\cdot n,\omega^{\delta}\cdot (n+1))$ such that 
 $B\cap f^{-1}[[\lambda_n, \lambda_{n+1})] = C\cap f^{-1}[[\lambda_n, \lambda_{n+1})]\in \I$ for each $n$, and 
 $B\cap f^{-1}[[\lambda_0,\lambda)] = C\cap f^{-1}[[\lambda_0,\lambda)]\notin \I$, where $\lambda = \sup_n \lambda_n$. Hence the proof is finished in this case.

\emph{Limit case:} $\alpha$ is a limit ordinal, so there is an increasing sequence  $(\alpha_n)_{n<\omega}$  in $\alpha$ such that  $\alpha_0=0$ and $\alpha=\sup_n\alpha_n$.

If $B\cap f^{-1}[[1,\omega^{\alpha_{n}})]\in \I$ for each $n<\omega$, the sequence $\lambda_n = \omega^{\alpha_n}$ works.
Now, assume there is $n<\omega$ with 
$C= B\cap f^{-1}[[1,\omega^{\alpha_{n}})]\notin \I$
If we apply the inductive hypothesis to the set $C$, we can obtain an increasing sequence 
 $(\lambda_n)_{n\in \omega}$ in $[1,\omega^{\alpha_n})$ such that 
 $B\cap f^{-1}[[\lambda_n, \lambda_{n+1})] = C\cap f^{-1}[[\lambda_n, \lambda_{n+1})]\in \I$ for each $n$, and 
 $B\cap f^{-1}[[\lambda_0,\lambda)] = C\cap f^{-1}[[\lambda_0,\lambda)]\notin \I$, where $\lambda = \sup_n \lambda_n$. Hence the proof is finished in this case.
\end{proof}

\begin{theorem}
\label{thm:char-conv}
     For any ideal $\I$ on $\omega$,   
    $$\conv_{1+\alpha}\not\leq_K \I \iff \omega^{\alpha}+1 \in \FinBW(\I).$$ 
In particular, if $\alpha$ is infinite, then 
    $\conv_{\alpha}\not\leq_K \I \iff \omega^{\alpha}+1 \in \FinBW(\I).$
\end{theorem}

\begin{proof}

($\implies$) Let $f: \omega \to  \omega^{\alpha}+1$. We need to find $D\notin \I$ such that the subsequence $f\restriction D$ is convergent.

If there exists $\xi \in \omega^{\alpha}+1$ such that $D = f^{-1}(\xi)\notin \I$, then 
$f\restriction D$ is convergent to $\xi$, so we are done.
Now, suppose $f^{-1}(\xi) \in \I$ for each  $\xi \in \omega^{\alpha}+1$.   

We define
$A_{\xi}=f^{-1}(\xi)$ 
for each $\xi \in \omega^{\alpha}+1$, 

Let $g: \omega \to \omega^{1+\alpha}+1$ be a function such that 
\begin{enumerate}
    \item 
$g[A_{\xi}] \subseteq [\omega \cdot \xi, \omega \cdot (\xi +1))$ for each $\xi < \omega^{\alpha}$, 
\item $g[A_{\omega^\alpha}] \subseteq \{\omega^{1+\alpha}\}$, and 
\item $g\restriction \omega \setminus A_{\omega^{\alpha}}$ is injective.
\end{enumerate}

Since $\conv_{1+\alpha}\nleq_K \I$, there exists a set  $B\notin \I$ such that $g[B]\in \conv_{1+\alpha}$. By Lemma~\ref{lem:P-like-property}, there exists an increasing sequence $(\lambda_n)_{n\in \omega}$ in $\omega^{\alpha}$ such that $ B\cap f^{-1}[[\lambda_n, \lambda_{n+1})]\in \I$ for each $n$ and $C=B\cap f^{-1}[[\lambda_0,\lambda)]\notin \I$, where $\lambda = \sup_n \lambda_n$. 
Since $(\lambda_n)_{n\in\omega}$ is an increasing sequence in $ \omega^{\alpha}$,
the sequence $(\omega \cdot \lambda_n)_{n\in\omega}$ is  increasing  in 
$\omega^{1+\alpha}$, 
so by Proposition~\ref{prop:gwiazdka},  there is $n_0\in \omega$ such that 
$g[B]\cap [\omega \cdot \lambda_n, \omega \cdot \lambda_{n+1})$ is finite for each $n\geq n_0$.
Since $g$ is injective on $\omega \setminus A_{\omega^{\alpha}}$, the set 
$$
B \cap g^{-1}[[\omega\cdot \lambda_n, \omega\cdot \lambda_{n+1})]
=
B \cap \bigcup_{\lambda_n \leq \xi < \lambda_{n+1}} A_{\xi}
$$ 
is also finite for each $n\geq n_0$. 

Let $$D=C \setminus \bigcup_{n<n_0} 
\left(\bigcup_{\lambda_n \leq \xi < \lambda_{n+1}} A_{\xi}\right).$$
Since   $C\notin I$ 
and for each $n$ 
$$C\cap \bigcup_{\lambda_n \leq \xi < \lambda_{n+1}} A_{\xi} = B\cap  f^{-1}[[\lambda_n,\lambda_{n+1})] \in \I,$$
 we conclude that $D\notin \I$. 
Once we show  that $f\restriction D$ is convergent to $\lambda$, the proof  will be finished.

Let $U$ be an open neighborhood of $\lambda$. Then there exists some $k_0$ such that $(\lambda_{k_0}, \lambda] \subset U.$ 
Notice that  
\begin{equation*}
    \begin{split}       
\{ n \in D : f(n) \notin U \}
&\subseteq 
\{n \in D : f(n) \leq \lambda_{k_0} \} 
\\&\subseteq   
D \cap \bigcup_{n<k_0} \left(\bigcup_{\lambda_n \leq \xi < \lambda_{n+1}} A_{\xi}\right)
=
\bigcup_{n<k_0}\left(D \cap \bigcup_{\lambda_n \leq \xi < \lambda_{n+1}} A_{\xi}\right).
    \end{split}
\end{equation*}
Since the set $D \cap \bigcup_{\lambda_n \leq \xi < \lambda_{n+1}} A_{\xi} $ is finite for each $n$, we get that $\{ n \in D : f(n) \notin U \}$ is finite as well.

\smallskip 

($\impliedby$) 
Suppose that $\conv_{1+\alpha} \leq_K \I $ and let $f: \omega \to  \omega^{1+\alpha}+1$ be a witnessing function, meaning that $B\notin \I $ implies that $ f[B]\notin \conv_{1+\alpha}$.  
Let $g:\omega^{1+\alpha}+1 \to  \omega^{\alpha}+1 $ be a function such that $g(\omega^{\alpha+1})=\omega^{\alpha}$ and 
$$g^{-1}(\xi)=[\omega \cdot \xi, \omega \cdot (\xi + 1))) \quad \text{ for each $\xi \in \omega^{\alpha}$.}$$
Let $h=g \circ f$.
 Then $h: \omega \to  \omega^{\alpha} +1 $, so 
once we show that the subsequence $h\restriction A$ is not convergent for any $A\notin \I$, 
the proof will be finished.
Suppose for the sake of contradiction that there exists some $A\notin \I $ such that $h\restriction A $ is convergent to some $\beta \in \omega^{\alpha} +1$.

\emph{Case (1).} $\beta $ is a successor ordinal or $\beta=0$.

Let $U=\{\beta\}$. As $U$ is a neighbourhood of $\beta$, there exists a finite set $K\subset \omega$ such that $h[A\setminus K]=\{\beta\}.$ Then $f[A\setminus K]\subset g^{-1}(\beta)=[\omega \cdot \beta, \omega \cdot \beta + \omega)$. But then $f[A\setminus K]^{d} \subset [\omega \cdot \beta, \omega \cdot \beta + \omega)^{d}=\{\omega \cdot \beta + \omega\}$. Then $f[A]\in \conv_{\alpha+1}$ (by Lemma ~\ref{lem:conv-ideal-via-derivative-set}), a contradiction with $f$ being the witness for $\conv_{1+\alpha} \leq_K \I $.

\emph{Case (2).} $\beta $ is a nonzero limit ordinal. 

Let $(\lambda_n)_{n<\omega}$  be any increasing sequence  in $\beta$ such that $\beta = \sup_n\lambda_n$.
We claim that the set  $f[A]\cap[\omega \cdot \lambda_n, \omega \cdot \lambda_{n+1})$ is finite for all $n$.

Suppose for the sake of contradiction that $f[A]\cap[\omega \cdot \lambda_n, \omega \cdot \lambda_{n+1})$ is infinite for some $n$. Then the set $A\cap f^{-1}[[\omega \cdot \lambda_n, \omega \cdot \lambda_{n+1})]$ is infinite. The set $V=(\lambda_{n+1},\beta]$ is an open neighborhood of $\beta$, but  
$$A\cap f^{-1}[[\omega \cdot \lambda_n, \omega \cdot \lambda_{n+1})]\subset \{i \in A : h(i)\notin V\},$$ 
which leads to a contradiction with $\beta$ being the limit of the subsequence $h\restriction A$. 

If we show that $f[A]^{d}$ is finite, then $f[A]\in \conv_{\alpha+1}$ (by Proposition~\ref{lem:conv-ideal-via-derivative-set}), which will lead to a contradiction with $f$ being the witness for $\conv_{\alpha+1} \leq_K \I$.

Since 
$$f[A]^d = 
(f[A]\cap [0,\omega \cdot \beta])^{d} 
\cup 
(f[A]\cap [\omega \cdot \beta,\omega \cdot (\beta+1)])^{d} 
\cup 
(f[A]\cap [\omega \cdot (\beta+1),\omega^{1+\alpha} ])^{d},$$
it is enough to show that sets from the above union are finite.

First, we show that $(f[A]\cap [0,\omega \cdot \beta])^{d}\subset \{\omega\cdot \beta\}.$
Take any $\gamma \in [0,\omega \cdot \beta)$. Then $\gamma < \omega \cdot \lambda_k$ for some $k$. Let $U=[0, \omega \cdot \lambda_k)$. Then $U$ is an open set containing $\gamma $. Moreover, $f[A]\cap U= \bigcup_{n<k} f[A]\cap[\omega \cdot \lambda_n, \omega \cdot \lambda_{n+1})$ which is finite. Hence $\gamma$ is not an accumulation point of $f[A]$.

Second, we observe that 
$$
(f[A]\cap [\omega \cdot \beta,\omega \cdot (\beta+1)])^{d} 
\subseteq [\omega \cdot \beta,\omega \cdot (\beta+1)]^{d} = \{\omega\cdot(\beta+1)\}. 
$$

Third, we show  that $(f[A]\cap [\omega \cdot (\beta+1),\omega^{1+\alpha} ])^{d}
=\emptyset$. 
Let $K=\{n \in A : h(n) > \beta \}$. Since $U=[0, \beta]$ is an open neighborhood of $\beta$, $K $ must be finite, so the set $f[A]\cap [\omega \cdot (\beta+1),\omega^{1+\alpha}] $ is finite as well, and consequently its derived set is empty.
\end{proof}

The \emph{disjoint union of topological spaces} $X_i$ for $i<n$ ($n<\omega$) is a topological space on the set 
$$\bigsqcup_{i<n} X_i = \bigcup_{i<n} X_i\times \{i\}$$
with open sets of the form 
$$\bigcup_{i<n} U_i\times\{i\} \quad \text{ for any open sets $U_i$ in $X_i$.}$$

\begin{lemma} 
\label{lem:disjoint-union-vs-FinBW}
    If $X\in \FinBW(\I)$, then $\bigsqcup_{i<n} X \in \FinBW(\I)$ for any $n<\omega$.
\end{lemma}

\begin{proof}
    Take any sequence $f:\omega \to \bigsqcup_{i<n} X$ and let $\pi : \bigsqcup_{i<n} X \to  X$ be the projection onto the first coordinate. 
    Then there exists $A\notin \I$ such that the subsequence  $(\pi \circ f)\restriction A$ is convergent to some $p\in X$. 
    Let's observe that $f^{-1}[X\times \{i\}]\cap A \notin I$ for some $i<n$. 
    Let  $C=f^{-1}[X\times \{i\}]\cap A \notin \I$.   We claim that the subsequence $f\restriction C$ is convergent to $(p,i)$. Take any open neighborhood $U$ of $(p,i)$. Then $\pi[U]$ is an open neighborhood of $\pi(p,i) = p$, so there exists some finite set $K $ such that $(\pi \circ f) [A\setminus K] \subseteq \pi [U]$, and  it follows that $f[C\setminus K]\subseteq U$.
\end{proof}

\begin{theorem}
\label{thm:conv-vs-FinBW}
Let  $\I$ be an ideal on $\omega$.
\begin{enumerate}
\item 
$\conv_{1+\alpha} \not\leq_K \I \implies \mathbb{K}_{\alpha} \subseteq \FinBW(\I).$

\item 
$ \conv_{(1+\alpha)+1} \leq_K \I \implies \FinBW(\I)\cap \mathbb{K}\subseteq \mathbb{K}_{\alpha}.$

\item 
$\conv_{(1+\alpha)+1} \leq_K \I \text{ and }  \conv_{1+\alpha} \not\leq_K \I \iff \FinBW(\I)\cap \mathbb{K}=\mathbb{K}_{\alpha}.$
\end{enumerate}
In particular, 
if $\alpha$ is infinite, then 
\begin{enumerate}
\item 
$\conv_{\alpha} \not\leq_K \I \implies \mathbb{K}_{\alpha} \subseteq \FinBW(\I)$.

\item 
$ \conv_{\alpha+1} \leq_K \I \implies \FinBW(\I)\cap \mathbb{K}\subseteq \mathbb{K}_{\alpha}$.

\item 
$\conv_{\alpha+1} \leq_K \I \text{ and }  \conv_{\alpha} \not\leq_K \I \iff \FinBW(\I)\cap \mathbb{K}=\mathbb{K}_{\alpha}$.
\end{enumerate}

\end{theorem}
\begin{proof}
(1)
    If $\conv_{1+\alpha} \not\leq_K \I$ then by Theorem ~\ref{thm:char-conv}, $\omega^{\alpha}+1 \in \FinBW(\I)$ which implies that  $\omega^{\beta}+1 \in \FinBW(\I)$ for any $\beta\leq \alpha $.
    Since 
    $\omega^{\beta}\cdot n+1$ is homeomorphic to a disjoint union  of finitely many copies of the space $\omega^{\beta}+1$, we can apply  Lemma ~\ref{lem:disjoint-union-vs-FinBW} to obtain $\omega^{\beta}\cdot n+1 \in \FinBW(\I)$ for any $\beta\leq \alpha $ and $n\in \omega$. Consequently $\mathbb{K}_{\alpha} \subset \FinBW(\I). $

(2)
    If $\conv_{(1+\alpha)+1} \leq_K \I $ then by Theorem ~\ref{thm:char-conv},  $  \omega^{\alpha+1}+1 \notin \FinBW(\I)$, which implies that $\omega^{\beta}\cdot n+1 \notin \FinBW(\I)$ for each $\beta\geq \alpha+1$ and $n\in \omega\setminus\{0\}$, which leads to $ \FinBW(\I)\cap \mathbb{K}\subseteq \mathbb{K}_{\alpha}$ by Theorem~\ref{thm:Mazurkiewicz-Sierpinski}.

(3)    
It follows from items (1) and (2) and Theorem ~\ref{thm:char-conv}.
\end{proof}

%%%%%%%%%%%%%%%%%%%%%%%%%%%%%%%%%%%%%%%%%%%%%%%%%%
%%%%% SECTION
%%%%%%%%%%%%%%%%%%%%%%%%%%%%%%%%%%%%%%%%%%%%%%%%%%

\section{Kat\v{e}tov order among the critical ideals}
\label{sec:Kat-among-cons}

\begin{proposition}
\label{prop:Katetov-between-conv}
 If $\alpha<\beta$, then $\conv_\beta\leq_K\conv_\alpha$.
\end{proposition}

\begin{proof}
It is enough to notice that the function $f:\omega^\alpha+1\to\omega^\beta+1$ given by $f(\xi)=\xi$ is a witness
for $\conv_\beta\leq_K\conv_\alpha$.
\end{proof}

\begin{lemma}
\label{lem:indecomposable}
Let $A,B\subseteq\omega^\alpha$.
If the order types of $A$ and $B$ are smaller than $\omega^\alpha$, then the order type of $A\cup B$ is smaller than $\omega^\alpha$, and 
consequently
$\omega^\alpha\setminus (A\cup B)\neq\emptyset$.
\end{lemma}

\begin{proof}
Suppose for the sake of contradiction that the order type of $A\cup B$ is $\omega^\alpha$.
Since $\omega^\alpha$ is an indecomposable ordinal, we obtain (see e.g.~\cite[Exercise 5 in Chapter I]{MR597342})  that  
    the order type of $A$ is $\omega^\alpha$ or the order type of $B$ is $\omega^\alpha$, a contradiction.
\end{proof}

\begin{definition}
    For every $\alpha$ and every ideal $\I$ on $\omega$ by $\conv_{\alpha+1}^\I$ we define  the ideal on $\omega^{\alpha+1}+1$ given by 
    \begin{equation*}
    \begin{split}
A\in \conv_{\alpha+1}^\I 
\iff 
& \{i\in\omega: |A\cap(\omega^\alpha\cdot i,\omega^\alpha\cdot (i+1)]|=\omega\}\in\I
\quad \text{and} 
\\&
A\cap(\omega^\alpha\!\!\cdot\! i,\omega^\alpha\!\!\cdot\! (i+1)]\in \conv((\omega^\alpha\!\!\cdot \! i,\omega^\alpha\!\!\cdot\! (i+1)])
\text{ for each } i\in\omega.
    \end{split}
\end{equation*}
\end{definition}

The following straightforward proposition reveals a relationship between the  ideal $\conv_{\alpha+1}$ and $\conv_{\alpha+1}^I$, and  provides a counterpart of Corollary~\ref{cor:conv-alpha-as-products}. 

\begin{proposition}\ 
\begin{enumerate}
\item $\conv_{\alpha+1}=\conv_{\alpha+1}^\fin \subseteq \conv_{\alpha+1}^\I$. 

\item $\conv_{\alpha+1}^\I$ is isomorphic to the ideal
$(\I\otimes\fin(\omega^\alpha+1)) \cap (\{\emptyset\}\otimes\conv_\alpha)$.
\end{enumerate}
\end{proposition}

\begin{lemma}
\label{lem:Katetov-for-conv-alpha}
    Let $\I$ be an ideal on $\omega$. If a function   $f:\omega^{\alpha+1}+1\to\omega^\alpha+1$ and an increasing sequence $(\lambda_n)_{n<\omega}$ are  such that $\lambda_0=0$, 
 $\sup\{\lambda_n:n<\omega\}=\omega^\alpha$ and 
 the order type of the set 
$f^{-1}[[\lambda_n,\lambda_{n+1})]$
is smaller than $\omega^\alpha$
for each $n<\omega$, then  there exists $A\notin \conv_{\alpha+1}^\I$ such that $f[A]\in \conv_\alpha$.
\end{lemma}

\begin{proof}
First, we notice that by Lemma~\ref{lem:indecomposable}, 
    $$[\omega^\alpha\cdot k,\omega^\alpha\cdot (k+1))\setminus \bigcup_{i\in F} f^{-1}[[\lambda_i,\lambda_{i+1})] \neq \emptyset$$
 for every finite set $F\subseteq\omega$ and $k\in \omega$. 
Now, we take any function $g:\omega\to\omega$ such that $g^{-1}(n)$ is infinite for each $n\in \omega$, and inductively  pick elements $a_n\in \omega^\alpha$ for $n\in \omega$ in such a way that 
$$a_n \in [\omega^\alpha\cdot g(n),\omega^\alpha\cdot (g(n)+1))\setminus \bigcup \left\{ f^{-1}[[\lambda_i,\lambda_{i+1})]: \exists j<n\, (f(a_j)\in [\lambda_i,\lambda_{i+1}))\right\}.$$
Then the set  $A=\{a_n:n\in \omega\}$ has the following properties:
\begin{enumerate}
    \item $A\cap [\omega^\alpha\cdot k,\omega^\alpha\cdot (k+1))$ is infinite for each $k\in \omega$,\label{asdfasdfsadfsdf}
    \item for each $n\in \omega$ there is $i\in \omega$ such that $f(a_n)\in [\lambda_i,\lambda_{i+1})$,\label{ierutoiwehgkjdslgkds}
    \item $|f[A]\cap [\lambda_i,\lambda_{i+1})|\leq 1$ for each $i\in \omega$.\label{kjierudfjsdfd}
\end{enumerate}
By property~(\ref{asdfasdfsadfsdf}), we get that $[\omega^\alpha\cdot k, \omega^\alpha\cdot (k+1)]\cap A^d\neq \emptyset$ for each $k\in \omega$, and consequently $A\notin \conv_{\alpha+1}^\I$. 
By properties~(\ref{ierutoiwehgkjdslgkds}) and (\ref{kjierudfjsdfd}), we get that 
$(f[A])^d=\{\omega^\alpha\}$, and consequently $f[A]\in \conv_\alpha$.
\end{proof}

\begin{theorem}
\label{thm:Katetov-between-conv}
$\conv_{\alpha}\not\leq_K\conv_{\alpha+1}^\I$ for every $\alpha$ and every ideal $\I$.
In particular, $\conv_{\alpha}\not\leq_K\conv_{\alpha+1}$ for every $\alpha$. \end{theorem}

\begin{proof}
    We proceed by induction on $\alpha$. 
First, we consider the successor case. We assume that 
$\conv_\alpha\not\leq_K\conv_{\alpha+1}^\I$ for every ideal $\I$ on $\omega$ and want to show that 
$\conv_{\alpha+1}\not\leq_K\conv_{\alpha+2}^\I$ for every ideal $\I$ on $\omega$.
Let $\lambda_n = \omega^\alpha\cdot n$ for each $n\in \omega$.
Then $(\lambda_n)_{n<\omega}$ is an increasing sequence in $\omega^{\alpha+1}$ such that $\lambda_0=0$ and $\sup\{\lambda_n:n<\omega\}=\omega^{\alpha+1}$.
Suppose for the sake of contradiction that $\conv_{\alpha+1}\leq_K\conv_{\alpha+2}^\I$ for some ideal $\I$ on $\omega$, and let $f:\omega^{\alpha+2}+1\to\omega^{\alpha+1}+1$ be a witness for this.
We have two cases.

\emph{Case (1).} The order type of $f^{-1}[[\lambda_n,\lambda_{n+1})]$ is smaller than $\omega^{\alpha+1}$ for each $n\in \omega$.

In this case, we can use Lemma~\ref{lem:Katetov-for-conv-alpha}, to obtain a set $A\notin \conv_{\alpha+2}^\I$ such that $f[A]\in \conv_{\alpha+1}$. But this contradicts  the fact that $f$ is a witness for $\conv_{\alpha+1}\leq_K\conv_{\alpha+2}^\I$.

\emph{Case (2).} The order type of $f^{-1}[[\lambda_n,\lambda_{n+1})]$ is greater  than or equal  to $\omega^{\alpha+1}$ for some  $n\in \omega$.

Let $A\subseteq f^{-1}[[\lambda_n,\lambda_{n+1})]$ be a set which has the order type equal to $\omega^{\alpha+1}$. Then the order type of $f[A]$ is at most  $\omega^{\alpha}$ as 
$$\ot(f[A]) \leq \ot([\lambda_n,\lambda_{n+1})) = \ot([\omega^\alpha\cdot n,\omega^\alpha\cdot(n+1))) = \omega^\alpha.$$
Then using the function $f\restriction A$, we can obtain 
a witness for $\conv_\alpha\leq_K\conv_{\alpha+1}$, a contradiction with the inductive  hypothesis.

Finally, we consider the limit case. We assume that $\alpha$ is a limit ordinal and $\conv_\beta\not\leq_K\conv_{\beta+1}^\I$ for each $\beta<\alpha$ and each ideal $\I$ on $\omega$.
Let $(\lambda_n)_{n<\omega}$ be an increasing sequence in $\omega^\alpha$ such that $\lambda_0=0$ and
$\sup\{\lambda_n:n<\omega\}=\omega^\alpha$.
Suppose for the sake of contradiction that $\conv_{\alpha}\leq_K\conv_{\alpha+1}^\I$ for some ideal $\I$ on $\omega$, and let $f:\omega^{\alpha+1}+1\to\omega^\alpha+1$ be a witness for this.
We have two cases.

\emph{Case (1).} The order type of $f^{-1}[[\lambda_n,\lambda_{n+1})]$ is smaller than $\omega^\alpha$ for each $n\in \omega$.

In this case, we can use Lemma~\ref{lem:Katetov-for-conv-alpha}, to obtain a set $A\notin \conv_{\alpha+1}^\I$ such that $f[A]\in \conv_\alpha$. But this contradicts  the fact that $f$ is a witness for $\conv_{\alpha}\leq_K\conv_{\alpha+1}$.

\emph{Case (2).} The order type of $f^{-1}[[\lambda_n,\lambda_{n+1})]$ is greater  than or equal  to $\omega^\alpha$ for some  $n\in \omega$.

Let $A\subseteq f^{-1}[[\lambda_n,\lambda_{n+1})]$ be a set which has the order type equal to $\omega^\alpha$. Then the order type of $f[A]$ is smaller than $\omega^\alpha$ as 
$$\ot(f[A]) \leq \ot([\lambda_n,\lambda_{n+1}))\leq \ot(\lambda_{n+1})=\lambda_{n+1}<\omega^\alpha.$$
Let $\beta<\alpha$ be such that $\lambda_{n+1}<\omega^\beta$.    
Then using the function $f\restriction A$, we can obtain 
a witness for $\conv_\beta\leq_K\conv_{\alpha}$.
Since $\beta+1<\alpha$, we get $\conv_\alpha\leq_K\conv_{\beta+1}$ by Corollary~\ref{cor:Katetov-between-conv}.
Consequently $\conv_\beta\leq_K\conv_{\beta+1}$, a contradiction with the inductive  hypothesis.
\end{proof}

\begin{corollary}
\label{cor:Katetov-between-conv}
 If $\alpha<\beta$, then $\conv_\beta\leq_K\conv_\alpha$  and $\conv_\alpha\not\leq_K\conv_\beta$.
\end{corollary}

\begin{proof}
It follows from Proposition~\ref{prop:Katetov-between-conv} and  Theorem~\ref{thm:Katetov-between-conv}. 
\end{proof}

\begin{lemma}
\label{lem:antichain}
Let $\I_0$ and $\I_1$ be ideals on $\omega$ such that $\I_0\not\leq_K\I_1$ and $\I_1$ is summable and tall. 
Let  $\I$ be  a tall ideal on a countable set $X$ such that 
$$\I\not\leq_K ((\I_1\restriction A)\otimes\Fin(X))\cap (\{\emptyset\}\otimes\I)$$
for every $A\notin\I_1$. 
Assume also that  $\fin\subseteq \I_0,\I_1$, $\fin(X)\subseteq \I$ and $X\notin \I$.
Then 
\[
(\I_0\otimes\Fin(X))\cap (\{\emptyset\}\otimes\I)\not\leq_K(\I_1\otimes\Fin(X))\cap (\{\emptyset\}\otimes\I).
\]
\end{lemma}

\begin{proof}
Denote $\J_i=(\I_i\otimes\Fin(X))\cap (\{\emptyset\}\otimes\I)$ for $i=0,1$. Let $f:\omega\times X\to\omega\times X$ be arbitrary. If there is $n\in\omega$ such that $f[\{n\}\times X]\cap(\{k\}\times X)$ is finite for all $k\in\omega$, then $C=\{n\}\times X\notin\J_1$, but $f[C]\in\J_0$, so we are done. Hence, assume that 
\[
E_n=\{k\in\omega: f[\{n\}\times X]\cap(\{k\}\times X)\text{ is infinite}\}
\]
is non-empty for all $n\in\omega$ and put $T=\{n\in\omega: E_n\text{ is infinite}\}$. Since $\I_0\not\leq_K\I_1$, either $\I_0\not\leq_K\I_1\restriction T$ or $\I_0\not\leq_K\I_1\restriction (\omega\setminus T)$.

Assume first that $\I_0\not\leq_K\I_1\restriction T$. In particular, $T$ is infinite in this case. For each $n\in T$ inductively find $g(n)\in\omega$ and $A_n\subseteq X$ such that:
\begin{itemize}
    \item[(a)] $g(n)\in E_n$,
    \item[(b)] $g(n)<g(n+1)$,
    \item[(c)] $f^{-1}[\{g(n)\}\times A_n]\cap(\{n\}\times X)$ is infinite,
    \item[(d)] $A_n\in\I$,
\end{itemize}
(note that (c) is possible to obtain as (a) implies that $f[\{n\}\times X]\cap(\{g(n)\}\times X)$ is infinite, while (d) is possible to obtain as $\I$ is tall). Since $g:T\to\omega$ and $\I_0\not\leq_K\I_1\restriction T$, there is $B\subseteq T$ such that $g[B]\in\I_0$, but $B\notin\I_1$. Then $$C=\bigcup_{n\in B}\{g(n)\}\times A_n\in \J_0$$ 
by $g[B]\in\I_0$ and items (b) and (d), while $f^{-1}[C]\notin \I_1\otimes\Fin(X)$ by $B\notin\I_1$ and item (c), so also $f^{-1}[C]\notin \J_1$.

Assume now that $\I_0\not\leq_K\I_1\restriction (\omega\setminus T)$. For each $n\in\omega\setminus T$ define
$$B_n=\left(f^{-1}[f[\{n\}\times X]\setminus(E_n\times X)]\right)_{(n)}.$$
For every $n\in\omega\setminus T$ the set $f[\{n\}\times X]\setminus(E_n\times X)$ belongs to $\J_0$, so if $B_n\notin \I$ for some $n$, then we are done (as $\J_1\not\ni\{n\}\times B_n\subseteq f^{-1}[f[\{n\}\times X]\setminus(E_n\times X)]$). Hence, assume that $B_n\in\I$ for every $n\in\omega\setminus T$.

Define $H_k=\{n\in\omega\setminus T: \max E_n\leq k\}$ for every $k\in\omega$. 

Consider first the case that $H_k\notin \I_1\restriction(\omega\setminus T)$ for some $k$. Observe that $\I\restriction (\omega\setminus B_n)$ is $\leq_K$-equivalent to $\I$, for every $n\notin T$ (as $B_n\in\I$), so 
$$((\I_1\restriction H_k)\otimes\Fin(X))\cap \left(\left(\{\emptyset\}\otimes\I\right)\restriction \left(\bigcup_{n\in H_k}\{n\}\times B_n\right)\right)$$ 
is $\leq_K$-equivalent to $((\I_1\restriction H_k)\otimes\Fin(X))\cap ((\{\emptyset\}\otimes\I)\restriction (H_k\times X))$. Since $\I\not\leq_K ((\I_1\restriction H_k)\otimes\Fin(X))\cap (\{\emptyset\}\otimes\I)$ (by the assumptions of this Lemma) and $\J_0\restriction ((k+1)\times X)$ is $\leq_K$-equivalent to $\I$, we can find $A\in\J_0\restriction ((k+1)\times X)$ such that $f^{-1}[A]\notin\J_1$.

Assume from now on that $H_k\in \I_1\restriction(\omega\setminus T)$ for every $k$. Let $\phi$ be a measure on $\omega\setminus T$ such that $\I_1\restriction(\omega\setminus T)=\{A\subseteq \omega\setminus T: \phi(A)<\infty\}$. For each $k$ find a finite set $G_k\subseteq H_k\setminus H_{k-1}$ such that 
$$\phi(G_k)\geq \phi(H_k\setminus H_{k-1})-\frac{1}{2^{k+1}}.$$ 

Define $G=\bigcup_{k\in\omega}G_k$ and observe that $\phi((\omega\setminus T)\setminus G)\leq 1$. Indeed, if $\phi((\omega\setminus T)\setminus G)>1$, then there should exist some finite $F\subseteq (\omega\setminus T)\setminus G$ such that $\phi(F)>1$. Since $\omega\setminus T=\bigcup_{k\in\omega}H_k$ and $H_k\subseteq H_{k+1}$, there is $m$ such that $F\subseteq H_m\setminus G=H_m\setminus\bigcup_{k\leq m}G_k$. But we have
\begin{equation*}
    \begin{split}
\phi(H_m)
&\geq 
\phi(F)+\phi\left(\bigcup_{k\leq m}G_k\right)
=
\phi(F)+\sum_{k\leq m}\phi(G_k)
\\&\geq
\phi(F)+\sum_{k\leq m} \left(\phi(H_k\setminus H_{k-1})-\frac{1}{2^{k+1}}\right)
\\&=
\phi(F)-\frac{2^{m+1}-1}{2^{m+1}}+\phi\left(\bigcup_{k\leq m}H_k\setminus H_{k-1}\right)
\\&=
\phi(F)-\frac{2^{m+1}-1}{2^{m+1}}+\phi(H_m),
    \end{split}
\end{equation*}
which gives us $\phi(F)\leq (2^{m+1}-1)/2^{m+1}<1$, a contradiction.

Since $(\omega\setminus T)\setminus G\in\I_1$, the restrictions $\I_1\restriction (\omega\setminus T)$ and $\I_1\restriction G$ are isomorphic (by \cite[Proposition 1.2]{MR3594409} and the fact that $\I_1$ is tall). 
Let $h:G\to\omega$ be given by $h\restriction G_k=k$. Since $\I_0\not\leq\I_1\restriction (\omega\setminus T)$, there is $B\in\I_0$ such that $g^{-1}[B]\subseteq G$ and $g^{-1}[B]\notin\I_1$. 

For every $k\in B$ and $n\in G_k$, using tallness of $\I$, find some $A_n\in\I$ such that $(f^{-1}[\{k\}\times A_n])_{(n)}$ is infinite and $\{k\}\times A_n\subseteq f[\{n\}\times X]$ (this is possible as $n\in G_k\subseteq H_k\setminus H_{k-1}$ means that $k\in E_n$). Define 
$$C=\bigcup_{k\in B}\bigcup_{n\in G_k} \{k\}\times A_n=\bigcup_{k\in B}\left(\{k\}\times\bigcup_{n\in G_k}A_n\right).$$ 
Then $C\in\J_0$, but $f^{-1}[C]\notin\J_1$. This finishes the proof.
\end{proof}

\begin{theorem}
For every $\alpha$, there are $2^\omega$ many pairwise $\leq_K$-incomparable $\Pi^0_5$ ideals that are above $\conv_{\alpha+1}$, but not above $\conv_\alpha$
in the Kat\v{e}tov order.
\end{theorem}

\begin{proof}
By \cite[Theorem 1]{MR3513296}  (see also \cite[Corollary 3.6]{MR3550610}), there is a family of tall summable ideals $\{\I_\beta:\beta<2^\omega\}$ such that  $\I_\beta\not\leq_K\I_\gamma$ for all distinct $\beta,\gamma<2^\omega$.

We claim that  the family of ideals $\{\conv_{\alpha+1}^{\I_\beta}:\beta<2^\omega\}$ is the required one.

For every $\beta<2^\omega$ the ideal $\conv_{\alpha+1}^{\I_\beta}$ is $\leq_K$-above $\conv_{\alpha+1}$ (as $\conv_{\alpha+1}\subseteq\conv_{\alpha+1}^{\I_\beta}$), but not $\leq_K$-above $\conv_\alpha$ (by Theorem \ref{thm:Katetov-between-conv}). Moreover, $\conv_{\alpha+1}^{\I_\beta}\not\leq_K\conv_{\alpha+1}^{\I_\gamma}$ for all distinct $\beta,\gamma<2^\omega$ (by Lemma \ref{lem:antichain}, which can be applied thanks to Theorem~\ref{thm:Katetov-between-conv} and the fact that $\conv_\alpha$ is tall).

Finally, observe that each $\conv_{\alpha+1}^{\I_\beta}$ is $\Pi^0_5$, since it is isomorphic to $(\{\emptyset\}\otimes\conv_\alpha)\cap(\I_\beta\otimes\Fin(\omega^\alpha+1))$ and $\{\emptyset\}\otimes\conv_\alpha$ is $\Pi^0_5$ (by \cite[Proposition~1.6.16]{alcantara-phd-thesis} and the fact that  $\conv_\alpha$ is $\Sigma^0_4$ by Proposition \ref{prop:Borel_complexity} and $\{\emptyset\}$ is $\Pi^0_1$), while $\I_\beta\otimes\Fin(\omega^\alpha+1)$ is $\Sigma^0_4$ (by \cite[Proposition~1.6.16]{alcantara-phd-thesis} and the fact that $\I_\beta$ and $\Fin(\omega^\alpha+1)$ are  $\Sigma^0_2$). 
\end{proof}

%%%%%%%%%%%%%%%%%%%%%%%%%%%%%%%%%%%%%%%%%%%%%%%%%%
%%%%% SECTION
%%%%%%%%%%%%%%%%%%%%%%%%%%%%%%%%%%%%%%%%%%%%%%%%%%

\section{Critical ideals versus the ideal \texorpdfstring{$\conv$}{conv} and  the space \texorpdfstring{$\omega_1$}{omega one}}
\label{sec:conv-alpha-vs-conv}

\begin{proposition}
\label{prop:conv-leq-Id}
 $\conv\leq_K \I_d$.   
\end{proposition}

\begin{proof}
    In \cite[Example 3]{MR1181163}, the author showed that $[0,1]\notin \FinBW(\I_d)$, so 
 $\conv\leq_K \I_d$ by Theorem~\ref{thm:Meza-compact-metric-versus-conv}. 
\end{proof}

Below, we consider $\omega_1$ as a topological space with the order topology.

\begin{proposition}
\label{prop:omega_1-not-in-FinBW-conv-alpha}
$\omega_1\in \FinBW(\conv)\setminus \FinBW(\conv_\alpha)$ for every  $\alpha$.
\end{proposition}

\begin{proof}
In \cite[Proposition~11.1(a)]{MR4584767}, the author showed that 
$\omega_1\in \FinBW(\I_d)$.
Consequently, $\omega_1\in \FinBW(\conv)$ by Proposition~\ref{prop:conv-leq-Id} and Theorem~\ref{thm:KAT-vs-FinBW}.
On the other hand,   the sequence $f:\omega^\alpha+1\to \omega_1$ given by $f(\xi)=\xi$ is a witness for $\omega_1\notin \FinBW(\conv_\alpha)$.
\end{proof}

\begin{corollary}
\label{cor:conv-vs-conv-alpha}
  $\conv \leq_K \conv_{\alpha}$ and  $\conv_\alpha\not\leq_K\conv$ for every  $\alpha$.
\end{corollary}

\begin{proof}
 Using Proposition~\ref{prop:omega_1-not-in-FinBW-conv-alpha} and Theorem~\ref{thm:KAT-vs-FinBW}, we obtain $\conv_\alpha\not\leq_K\conv$.
Now, suppose for the sake of contradiction that  
    $\conv \not\leq_K \conv_{\alpha}$ for some $\alpha$.
Then $[0,1]\in \FinBW(\conv_\alpha)$ by Theorem~\ref{thm:Meza-compact-metric-versus-conv}.
Since  the space  $\omega^{\alpha+1}+1$ is homeomorphic to a closed 
subset of $[0,1]$,
we obtain 
$\omega^{\alpha+1}+1\in \FinBW(\conv_\alpha)$.
Then by Theorem~\ref{thm:char-conv}, we obtain   
$\conv_{1+(\alpha+1)}\not\leq_K \conv_{\alpha}$, which in turn contradicts  Corollary~\ref{cor:Katetov-between-conv}.
\end{proof}

\begin{corollary}
\label{cor:omega_1}
\ 
\begin{enumerate}
    \item 
For any ideal $\I$, 
    $$\omega_1\in \FinBW(\I) \iff \conv_\alpha\not\leq_K\I \text{ for any $\alpha$.}$$

\item There is \emph{no} single  ideal $\I_{\omega_1}$ such that 
$$
\omega_1\in \FinBW(\I) \iff \I_{\omega_1}\not\leq_K \I.
$$
for any ideal $\I$.
\end{enumerate}
\end{corollary}

\begin{proof}
    (1, $\implies$)
Suppose that $\conv_\alpha\leq_K\I$ for some $\alpha$.
Then $\omega_1\in \FinBW(\I)\subseteq \FinBW(\conv_\alpha)$ by Theorem~\ref{thm:KAT-vs-FinBW}, a contradiction with Proposition~\ref{prop:omega_1-not-in-FinBW-conv-alpha}.

    (1, $\impliedby$)
    Suppose that $\omega_1\notin \FinBW(\I)$.
    Then there is $f:\omega\to \omega_1$ such that $f\restriction A$ is not convergent for any $A\notin \I$.
    Since $\omega_1$ has uncountable cofinality, there is a countable  $\alpha$  such that $f[\omega]\subseteq\omega^\alpha+1$. Then $f:\omega\to \omega^\alpha+1$, so $f$ is a witness for $\omega^\alpha+1\notin \FinBW(\I)$.
    Thus $\conv_{1+\alpha}\leq_K\I$ by Theorem~\ref{thm:char-conv}.

(2)
Suppose for the sake of contradiction that there is an ideal $\I_{\omega_1}$
such that 
$\omega_1\in \FinBW(\I) \iff \I_{\omega_1}\not\leq_K \I$ for any ideal $\I$.
Then $\omega_1\notin \FinBW(\I_{\omega_1})$, so by item (1) we find $\alpha_0$ with $\conv_{\alpha_0}\leq_K \I_{\omega_1}$.
We claim that $\I_{\omega_1}\not\leq_K \conv_{\alpha_0+1}$.
Indeed, otherwise we would obtain 
$\I_{\omega_1} \leq_K \conv_{\alpha_0+1} \leq_K\conv_{\alpha_0}\leq_K\I_{\omega_1}$, so $\conv_{\alpha_0}\leq_K\conv_{\alpha_0+1}$, a contradiction with Theorem~\ref{thm:Katetov-between-conv}.
Now, since $\I_{\omega_1}\not\leq_K \conv_{\alpha_0+1}$, we can use our assumption to obtain  $\omega_1\in \FinBW(\conv_{\alpha_0+1})$. 
Then by item (1) with $\alpha=\alpha_0+1$ we get 
$\conv_{\alpha_0+1}\not\leq_K \conv_{\alpha_0+1}$, a contradiction.
\end{proof}

\begin{lemma}
\label{lem:P-ideals}
If $\I$ and $\J$ are tall ideals, $\J$ is a P-ideal and $\J\not\leq_K\I$ then $\J\not\leq_K \I\otimes\Fin$.
\end{lemma}

\begin{proof}
Let $f:\omega\times \omega\to\omega$ be arbitrary. We need to find $A\in\J$ such that $f^{-1}[A]\notin\I\otimes\Fin$. Denote
\[
T=\{n\in\omega: f\restriction (\{n\}\times\omega)\text{ is finite-to-one}\}.
\]
There are two possibilities: either $T\in \I$ or $T\notin\I$. 

If $T\in\I$, for each $n\in\omega\setminus T$ find an infinite $A_n\subseteq\{n\}\times\omega$ and $g(n)\in\omega$ such that $f[A_n]=\{g(n)\}$. Since $\I$ and $\I\restriction (\omega\setminus T)$ are isomorphic (by \cite[Proposition 1.2]{MR3594409} as $T\in\I$ and $\I$ is tall)  
and $\J\not\leq_K\I$, there is $A\in\J$ such that $g^{-1}[A]\notin\I\restriction (\omega\setminus T)$. Then 
$$f^{-1}[A]\supseteq\bigcup_{n\in g^{-1}[A]}A_n\notin\I\otimes\Fin,$$ 
so $A$ is the required set.

If $T\notin\I$, using tallness of $\J$, for each $n\in T$ find infinite $B_n\in\J$ such that $B_n\subseteq f[\{n\}\times\omega]$. Since $\J$ is a P-ideal,  there is $A\in\J$ such that $B_n\setminus A\in\Fin$ for all $n\in\omega$. In particular, 
$$f^{-1}[A]\cap(\{n\}\times\omega)\supseteq f^{-1}[A\cap B_n]\cap(\{n\}\times\omega)\notin\Fin$$ 
for all $n\in T$. Hence, $f^{-1}[A]\notin\I\otimes\Fin$.
\end{proof}

\begin{theorem}
There are $2^\omega$ many pairwise $\leq_K$-incomparable $\Sigma^0_4$ ideals that are above $\conv$, but not above any $\conv_\alpha$
in the Kat\v{e}tov order.
\end{theorem}

\begin{proof}
By \cite[Theorem 1]{MR3513296} (see also [Corollary 3.6]), there is a family of ideals $\{\I_\beta:\beta<2^\omega\}$ such that each $\I_\beta$ is a tall summable ideal and $\I_\beta\not\leq_K\I_\gamma$ for all distinct $\beta,\gamma<2^\omega$.

Define $\J_\beta=(\I_\beta\otimes\Fin)\cap (\{\emptyset\}\otimes\I_d)$ for all $\beta<2^\omega$. 

Now we show that $\conv_\alpha\not\leq_K\J_\beta$ for all $\beta<2^\omega$ and all $\alpha$. Indeed, if some $\J_\beta$ would be $\leq_K$-above some $\conv_\alpha$, then also $\J_\beta\restriction (\{0\}\times\omega)$ would be $\leq_K$-above $\conv_\alpha$. However, $\J_\beta\restriction (\{0\}\times\omega)$ is isomorphic to $\I_d$, so it would contradict Corollary \ref{cor:omega_1} and \cite[Proposition~11.1(a)]{MR4584767} (where it is shown that $\omega_1\in\FinBW(\I_d)$). Hence, $\conv_\alpha\not\leq_K\J_\beta$.

To show that $\conv\leq_K\J_\beta$ for all $\beta<2^\omega$, let $(I_n)_{n\in\omega}$ be a sequence of pairwise disjoint closed subintervals of $[0,1]$ with rational endpoints and such that $(I_n)_{n\in\omega}$ converges to $0$ (that is, if $x_n\in I_n$ for all $n$, then $\lim_n x_n=0$). For each $n\in\omega$, since $\conv\restriction (I_n\cap\mathbb{Q})$ and $\conv$ are isomorphic and $\conv\leq_K \I_d$ (by Proposition \ref{prop:conv-leq-Id}), there is $g_n:\omega\to I_n\cap\mathbb{Q}$ witnessing $\conv\restriction (I_n\cap\mathbb{Q})\leq_K\I_d$. Let $f:\omega\times \omega \to[0,1]\cap\mathbb{Q}$ be given by $f(n,k)=g_n(k)$. Then $f$ witnesses $\conv\leq_K\J_\beta$.

Now we prove that $\J_\beta\not\leq_K\J_\gamma$ for all distinct $\beta,\gamma<2^\omega$. This will follow from Lemma \ref{lem:antichain} once we show that all its assumptions are met. It is known that $\I_d$ is a tall P-ideal and $\I_d\not\leq_K\cK$ for every summable ideal $\cK$ (by
Proposition~\ref{prop:conv-leq-Id} and 
\cite[Corollary~3.14]{MR3692233} where the authors proved that $\conv\not\leq_K\I$ for any $F_\sigma$ ideal $\I$). 
Hence, $\I_d\not\leq_K\cK\otimes\Fin$ for every summable ideal $\cK$ (by Lemma \ref{lem:P-ideals}). Since $\I_\gamma\restriction A$ is a summable ideal, for every $A\notin\I_\gamma$, we get that $\I_d\not\leq_K((\I_\gamma\restriction A)\otimes\Fin)\cap (\{\emptyset\}\otimes\I_d)$ (since $((\I_\gamma\restriction A)\otimes\Fin)\cap (\{\emptyset\}\otimes\I)\subseteq (\I_\gamma\restriction A)\otimes\Fin$). 

Finally, observe that $\{\emptyset\}\otimes\I_d$ is $\Pi^0_3$ 
by \cite[Proposition~1.6.16]{alcantara-phd-thesis} and the fact that
$\I_d$ is $\Pi^0_3$ and $\{\emptyset\}$ is $\Pi^0_1$, while $\I_\beta\otimes\Fin$ is $\Sigma^0_4$ 
by \cite[Proposition~1.6.16]{alcantara-phd-thesis} and the fact that
$\I_\beta$ and $\Fin$ are  $\Sigma^0_2$.
\end{proof}

%%%%%%%%%%%%%%%%%%%%%%%%%%%%%%%%%%%%%%%%%%%%%%%%%%
%%%%% SECTION
%%%%%%%%%%%%%%%%%%%%%%%%%%%%%%%%%%%%%%%%%%%%%%%%%%

\section{Additional  properties of critical ideals}
\label{sec:properties-of-critical-ideals}

%%%%%%%%%%%%%%%%%%%%%%%%%%%%%%%%%%%%%%%%%%%%%%%%%%
%%%%% SECTION
%%%%%%%%%%%%%%%%%%%%%%%%%%%%%%%%%%%%%%%%%%%%%%%%%%

\subsection{Critical ideals have the property KAT}

We say that an ideal $\I$ on $X$  \emph{contains an isomorphic copy} of  an ideal $\J$ on $Y$ (in short $\J\sqsubseteq\I$) if there is a bijection  $f:X\to Y$ such that  $f^{-1}[A]\in\I$ 
for every $A\in \J$. 
In  \cite[Lemma~3.3]{MR3034318}, the authors showed that 
if $\J$ is tall, then we can only require that $f$ is one-to-one in the definition of $\sqsubseteq$.

    \begin{proposition}
    \label{prop:KAT-for-conv}
        $\conv_{\alpha} \sqsubseteq \conv_{\alpha} \otimes \{\emptyset\}$ for every  $\alpha$.
    \end{proposition}
    \begin{proof}
        Let $f: (\omega^{\alpha}+1) \times \omega \to   \omega^{\alpha}+1$ be a one-to-one function such that  
 $f[\{\omega^{\alpha}\}\times \omega]=[0,\omega)$
 and 
for each $\xi \in \omega^{\alpha} $ we have
$$f[\{\xi\}\times \omega] =  [\omega \cdot (1+\xi), \omega \cdot (1+\xi + 1)).$$ 
We claim that $f$ is a witness for $\conv_\alpha\sqsubseteq\conv_\alpha\otimes\{\emptyset\}$.
Since the ideal $\conv_\alpha$ is tall
(by Lemma~\ref{lem:conv-ideal-via-derivative-set}(\ref{lem:conv-ideal-via-derivative-set:item-4})) and the function  $f$ is one-to-one, we only need to show that $f^{-1}[A]\in \conv_\alpha\otimes\{\emptyset\}$ for every $A\in \conv_\alpha$.
Take any $A\in \conv_{\alpha}$ and suppose for the sake of contradiction that $f^{-1}[A]\notin \conv_\alpha\otimes\{\emptyset\}$.
Then 
$$B =  \{\xi\in \omega^\alpha+1: f^{-1}[A] \cap (\{\xi\}\times \omega)\neq\emptyset\}\notin \conv_\alpha.$$
By Proposition~\ref{prop:gwiazdka}, there exists an increasing sequence $(\lambda_n)_{n\in \omega}$ in $\omega^\alpha$ 
and an infinte set $C\subseteq\omega$ such that 
the intersection $B\cap [\lambda_n,\lambda_{n+1})$ is infinite for each  $n\in C$.
Then for each $n\in C$ we can find $\tau_n\in [\omega\cdot(1+\lambda_n),\omega\cdot(1+\lambda_{n+1}+1)]\cap A^d$.
Consequently, $A^d$ is infinite, so $A\notin \conv_\alpha$, a contradiction.
\end{proof}

The following corollary shows that the ideals $\conv_\alpha$ have the property $Kat$ i.e.~we can replace arbitrary function by  a bijection or a finite-to-one function when comparing the ideals $\conv_\alpha$ with other ideals in the Kat\v{e}tov order (this phenomenon for arbitrary ideals was introduced and examine in details in \cite{MR3034318}). 

We say that an ideal $\J$ on $Y$  is below an ideal $\I$ on $X$  in  the \emph{Kat\v{e}tov-Blass order} (in short $\J\leq_{KB}\I$) if there is a finite-to-one function $f:X\to Y$ such that $f^{-1}[A]\in\I$ for every $A\in \I$.

    \begin{corollary}
    For any ideal $\I$, 
    $$ \conv_\alpha \leq_K \I \iff \conv_\alpha\leq_{KB} \I \iff \conv_\alpha \sqsubseteq \I.$$
    \end{corollary}

\begin{proof}
It follows from Proposition~\ref{prop:KAT-for-conv} and \cite[Theorem~3.4]{MR3034318}.
\end{proof}

Using Proposition~\ref{prop:conv-for-1-2-3}, we can see that the above corollary extends \cite[Theorem~6.2]{MR2899832} and \cite[Proposition~4.4]{MR4584767}, where the authors proved the above corollary for $\Fin\otimes\Fin\approx \conv_2$ and $\BI\approx\conv_3$, respectively.

%%%%%%%%%%%%%%%%%%%%%%%%%%%%%%%%%%%%%%%%%%%%%%%%%%
%%%%% SECTION
%%%%%%%%%%%%%%%%%%%%%%%%%%%%%%%%%%%%%%%%%%%%%%%%%%

\subsection{Borel complexity and \texorpdfstring{$P^-$}{P-} property}

An ideal $\I$ on $X$ is $P^-$ (a.k.a. \emph{hereditary weak P}) if for every partition $\cA$ of  any set $C\in \I^+$ into sets from $\I$ there exists  $B\in \I^+$ such that $B\subseteq C$ and  $B\cap A$ is finite for each $A\in \cA$ (see \cite[p.~2030]{MR3692233} and \cite[Definition~4.8]{MR4584767}, resp.).
It is known that $\conv$ is not a $P^-$ ideal (see e.g.~\cite[proof of Proposition 4.10(b)]{MR4584767}), however there are $P^-$ ideals which are above $\conv$ in the Kat\v{e}tov order
(for instance, $\I_d$ is $P^-$ and $\conv\leq_K\I_d$ by Proposition~\ref{prop:conv-leq-Id}).
Every ideal $\conv_\alpha$ is above the ideal $\conv$ in the Kat\v{e}tov order, however there is no $P^-$ ideal above any $\conv_\alpha$ ideal as shown by the following proposition.

\begin{proposition}
\label{prop:Pminus-versus-conv-alpha}
If  $\I$ is a $P^-$ ideal, then $\conv_{\alpha} \nleq_K \I$ for every $\alpha$.
In particular, $\conv_\alpha$ is not  $P^-$ for any $\alpha$.
\end{proposition}

\begin{proof}
If $\I$ is a $P^-$ ideal, then $\omega_1\in\FinBW(\I)$ by \cite[Proposition~6.1]{MR4584767}. Thus, Corollary \ref{cor:omega_1} finishes the proof.
\end{proof}

\begin{corollary}
\label{cor:conv-alpha-is-not-Pi-0-4}
$\conv_\alpha\in \Sigma^0_4\setminus \Pi^0_4$ for every $\alpha$.
\end{corollary}

\begin{proof}
In \cite[Proposition~4.9]{MR4584767}, the author proved that every $\Pi^0_4$ ideal is $P^-$, so $\conv_\alpha$ is not $\Pi^0_4$ by Proposition~\ref{prop:Pminus-versus-conv-alpha}. On the other hand, 
$\conv_\alpha$ is $\Sigma^0_4$ by Proposition \ref{prop:Borel_complexity}.
\end{proof}

%%%%%%%%%%%%%%%%%%%%%%%%%%%%%%%%%%%%%%%%%%%%%%%%%%
%%%%% SECTION
%%%%%%%%%%%%%%%%%%%%%%%%%%%%%%%%%%%%%%%%%%%%%%%%%%

\subsection{Cardinal characteristics of critical ideals}
\label{subsec:cardinal-characteristics}

Some properties of ideals can be described by cardinal characteristics associated with them. There are four well known cardinal characteristics called additivity, covering, uniformity and cofinality defined in the following way
for an ideal $\I$ on $X$ (see e.g.~\cite{MR1350295}):
$\add(\I) =\min\left\{|\cA|:\cA\subseteq\I\land \bigcup\cA\notin\I\right\}$,
$\cov(\I)  = \min\left\{|\cA|: \cA\subseteq\I\land \bigcup\cA=X\right\}$, 
$\non(\I)  = \min\{|A|:A\notin \I\}$, 
$\cof(\I)  =\min\{|\cA|:\cA\subseteq\I\land \forall B\in\I \, \exists A\in\cA \, (B\subseteq A)\}$.
These  characteristics are useful in the case of ideals on an uncountable set $X$ (for instance in the case of the $\sigma$-ideal  of all meager sets and the $\sigma$-ideal of all Lebesgue null sets). However, they are (but $\cof$) useless in the case of ideals on countable sets as  $\add(\I)=\cov(\I)=\non(\I)=\aleph_0$ for every ideal $\I$ on a countable set.
Fortunately, 
Hern\'{a}ndez and Hru\v{s}\'{a}k
introduced in \cite{MR2319159} (see also~\cite{MR2777744}) certain versions of these characteristics more suitable for tall ideals on countable sets
 ($A\subseteq^* B$ means that $A\setminus B$ is finite in this definitions):
\begin{equation*}
\begin{split}
\adds(\I) & =\min\{|\cA|:\cA\subseteq\I\land \neg \exists B\in\I \, \forall A\in\cA \, (A\subset^* B)\}, \\
\covs(\I)  & = \min\{|\cA|: \cA\subseteq\I\land \forall B\in [\omega]^{\aleph_0} \, \exists A\in\cA \, (|A\cap B|=\aleph_0)\},\\
\nons(\I) & = \min\{|\cA|:\cA\subseteq [\omega]^\omega\land \forall B\in\I \, \exists A\in\cA \, (|A\cap B|<\aleph_0)\}, \\
\cofs(\I) & =\min\{|\cA|:\cA\subseteq\I\land \forall B\in\I \, \exists A\in\cA \, (B\subset^* A)\}.
\end{split}
\end{equation*}
There are some inequalities holding among these characteristics for all ideals (see also Figure~\ref{fig:cardinal-characteristics}):
$\aleph_0\leq \adds(\I)\leq \covs(\I)\leq \cofs(\I)\leq 2^{\aleph_0}$
and
$\aleph_0\leq \adds(\I)\leq \nons(\I)\leq \cofs(\I)\leq 2^{\aleph_0}$  (see e.g.~\cite[p.~578]{MR2777744}).

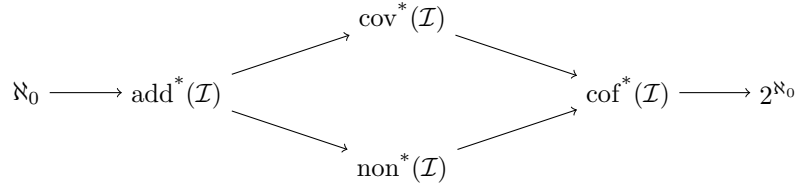
\begin{figure}
\centering
\begin{tikzpicture}[scale=1]
  \node (omega) at (-5,0) {$\aleph_0$};
  \node (add) at (-3,0) {$\adds(\I)$};
  \node (cov) at (0,1) {$\covs(\I)$};
  \node (non) at (0,-1) {$\nons(\I)$};
  \node (cof) at (3,0) {$\cofs(\I)$};
  \node (c) at (5,0) {$2^{\aleph_0}$};

  \draw[->] (omega) -- (add); 
  \draw[->] (add) -- (cov);
  \draw[->] (cov) -- (cof);
  \draw[->] (add) -- (non);
  \draw[->] (non) -- (cof);
  \draw[->] (cof) -- (c);
\end{tikzpicture}
    \caption{Relationships between chardinal characteristics for ideals on $\omega$ ($\kappa\to\lambda$ means $\kappa\leq \lambda$ in this diagram).}
    \label{fig:cardinal-characteristics}
\end{figure}

It is known (see e.g.~\cite{MR2777744}) that 
$\adds(\conv)=\adds(\Fin\otimes\Fin)=\aleph_0$,
$\nons(\conv)=\nons(\Fin\otimes\Fin)=\aleph_0$,
$\covs(\conv)=\cofs(\conv)=2^{\aleph_0}$, 
$\covs(\Fin\otimes\Fin)=\bnumber$ and  $\cofs(\Fin\otimes\Fin)=\dnumber$, where
$\bnumber$  is the \emph{bounding number} i.e.~the smallest cardinality of any unbounded family in the poset $(\omega^\omega,\leq^*)$ and 
$\dnumber$ is the \emph{dominating number} i.e.~the smallest cardinality of
any dominating family in the poset $(\omega^\omega,\leq^*)$ (for more on these cardinals see for instance \cite{MR2768685}). 
Moreover, if $\I\leq_K\J$, then  $\covs(\I)\geq \covs(\J)$ and $\nons(\I)\leq\nons(\J)$ (\cite[Proposition~3.1]{MR2319159}).

The following theorem provide the values of the above mentioned characteristics for critical ideals considered in the paper.
\begin{proposition} 
For every $\alpha$, 
$$\nons(\conv_\alpha) = \adds(\conv_\alpha)= \aleph_0, \quad 
\covs(\conv_\alpha)= \bnumber\quad\text{and} \quad \cofs(\conv_\alpha)=\dnumber.$$
\end{proposition}

\begin{proof}
($\nons(\conv_\alpha)=\aleph_0$)
Since $\conv_\alpha\leq_K\conv_2\leq_K\Fin\otimes\Fin$, we obtain $\aleph_0\leq \nons(\conv_\alpha)\leq \nons(\Fin\otimes\Fin) = \aleph_0$.

($\adds(\conv_\alpha)=\aleph_0$) Since $\aleph_0\leq \adds(\conv_\alpha) \leq \nons(\conv_\alpha)=\aleph_0$, it follows that $\adds(\conv_\alpha)=\aleph_0$.

($\covs(\conv_\alpha)=\bnumber$)  Since $\conv_\alpha\leq_K\conv_2\leq_K\Fin\otimes\Fin$, then $\covs(\conv_\alpha)\geq \covs(\fin\otimes\fin)=\bnumber$. We will show the reverse inequality by induction on $\alpha$.
Assume the statement holds for any $\beta<\alpha$.

\emph{Successor case:} $\alpha$ is a successor ordinal, so $\alpha=\gamma+1$.
Let's notice that for any infinite set $X\in \conv_{\gamma+1}$, either there is some $n $ such that the set $X\cap \omega^\gamma \cdot n$ is infinite, or the set $X\cap[\omega^\gamma \cdot n, \omega^\gamma \cdot (n+1)]$ is finite for each $n$. 
Since  $\conv_\gamma$ is isomorphic to $\conv(\omega^\gamma \cdot n)$, by inductive hypothesis we obtain 
a witness, say $\cA_n\subseteq \conv(\omega^\gamma \cdot n)$,  for $\covs(\conv(\omega^\gamma \cdot n)) =\bnumber$ for each $n\geq1$. 
Since $\covs(\{\emptyset\}\otimes \fin)=\bnumber$, there exists a family $\cB\subset \cP(\omega^{\gamma+1}+1)$ with the following properties: 
\begin{enumerate}
    \item $|\mathcal{B}|=\bnumber$,
    \item if $F\in \cB$ then the set $F\cap [\omega^\gamma \cdot n,\omega^\gamma \cdot (n+1))$ is finite for each $n$,
    \item for any infinite $A\subset \omega^{\gamma+1}+1$, if the set $A\cap[\omega^\gamma \cdot n,\omega^\gamma \cdot (n+1))$ is finite for each $n$, then there exists some $F\in \cB$ such that the set $A\cap F$ is infinite.
\end{enumerate}
Let $\cC=\cB \cup \bigcup_{n<\omega} \cA_n$. Then $\cC$ is a subfamily of $\conv_{\gamma+1}$ of cardinality  $\bnumber$ having the property that for any infinite $D\in \conv_{\gamma+1}$ there exists some $C\in \cC$ such that the set $D\cap C$ is infinite, so it follows that $\covs(\conv_{\gamma+1})\leq\bnumber$. 

\emph{Limit case:}  $\alpha$ is a limit ordinal. Let $(\alpha_n)_{n<\omega}$ be a strictly increasing sequence that is cofinal in $\alpha$ with $\alpha_0=0$. Let's notice that for any infinite set $X\in \conv_{\alpha}$, either there is some $n $ such that the set $X\cap \omega^{\alpha_n}$ is infinite, or the set $X\cap[\omega^{\alpha_n}, \omega^{\alpha_{n+1}}]$ is finite for each $n$. 
By inductive hypothesis, we can find a witness, say $\cA_n\subseteq \conv_{\alpha_n}$, for $\covs(\conv_{\alpha_n}) = \bnumber$
 for each $n\geq 1$.  
Since $\covs(\{\emptyset\}\otimes \fin)=\bnumber$, there exists a family $\cB\subset \cP(\omega^{\alpha}+1)$ with the following properties: 
\begin{enumerate}
    \item $|\mathcal{B}|=\bnumber$,
    \item if $F\in \cB$ then the set $F\cap [\omega^{\alpha_n},\omega^{\alpha_{n+1}})$ is finite for each $n$,
    \item for any infinite $A\subset \omega^{\alpha}+1$, if the set $A\cap[\omega^{\alpha_n},\omega^{\alpha_{n+1}})$ is finite for each $n$, then there exists some $F\in \cB$ such that the set $A\cap F$ is infinite.
\end{enumerate}
Analogically to the previous case, the family $\cC=\cB \cup \bigcup_{n<\omega} \cA_n$ is a witness for the required inequality.

($\cofs(\conv_\alpha)=\dnumber$) First, we prove that $\cofs(\conv_\alpha)\leq\dnumber$ by induction on $\alpha$. Assume the statement holds for any $\beta<\alpha$.

\emph{Successor case:} $\alpha$ is a successor ordinal, so $\alpha=\gamma+1$. Let's notice that for any set $A\in \conv_{\gamma+1}$  there is some $n$ such that $A\cap\omega^\gamma \cdot n \in \conv(\omega^\gamma \cdot n) $ and for any $k\geq n$ the set $A\cap[\omega^\gamma\cdot k, \omega^\gamma\cdot (k+1))$ is finite.
Since  $\conv_\gamma$ is isomorphic to $\conv(\omega^\gamma \cdot n)$, by inductive hypothesis we obtain a witness, say $\mathcal{A}_n\subseteq \conv(\omega^\gamma \cdot n)$, for $\cofs(\conv(\omega^\gamma \cdot n)) =\dnumber$  for each $n\geq1$.
Since $\cofs(\{\emptyset\}\otimes \fin)=\dnumber$, there exists a family $\cB\subseteq \cP(\omega^{\gamma+1}+1)$ with the following properties: 
\begin{enumerate}
    \item $|\cB|=\dnumber$,
    \item if $F\in \cB$ then the set $F\cap [\omega^\gamma \cdot n,\omega^\gamma \cdot (n+1))$ is finite for each $n$,
    \item for any $A\subset \omega^{\gamma+1}+1$, if the set $A\cap[\omega^\gamma \cdot n,\omega^\gamma \cdot (n+1))$ is finite for each $n$ then there exists some $F\in \cB$ such that $A\subset^* F$.
\end{enumerate}
Let $\cC=\{A\cup B : A\in \bigcup_{n\in \omega}\cA_n, B\in \cB\}$. Then $\cC$ is a subfamily of $\conv_{\gamma+1}$ such that $|\cC|=\dnumber$ and for any $D\in \conv_{\gamma+1}$ there exists some $C\in \cC$ with $D\subset^*C$,  so it follows that $\cofs(\conv_{\gamma+1})\leq\dnumber$. 

\emph{Limit  case:} $\alpha$ is a limit ordinal. Let $(\alpha_n)_{n<\omega}$ be a strictly increasing sequence that is cofinal in $\alpha$ with $\alpha_0=0$. Then 
 $$\conv_{\alpha}=\bigcup_{n\in \omega}\left\{A\subseteq \omega^{\alpha}+1 : A\cap\omega^{\alpha_n} \in \conv_{\alpha_n} \ \land \ \forall_k\geq n \ A\cap[\omega^{\alpha_k}, \omega^{\alpha_{k+1}})\in\Fin \right\}.$$
 Let $\mathcal{A}_n\subset \conv_{\alpha_n}$ be a witness for $\cofs(\conv_{\alpha_n}) =\dnumber$. Since $\cofs(\{\emptyset\}\otimes \fin)=\dnumber$, there exists a family $\mathcal{B}\subseteq \cP(\omega^{\alpha}+1)$ with the following properties: 
\begin{enumerate}
    \item $|\mathcal{B}|=\dnumber$,
    \item if $F\in \mathcal{B}$ then the set $F\cap [\omega^{\alpha_n}, \omega^{\alpha_{n+1}})$ is finite for each $n$,
    \item for any $A\subset \omega^{\alpha}+1$, if the set $A\cap[\omega^{\alpha_n}, \omega^{\alpha_{n+1}})$ is finite for each $n$, then there exists some $F\in \mathcal{B}$ such that $A\subset^*F$.
\end{enumerate}
Analogically to the previous case, the family $\cC=\{A\cup B : A\in \bigcup_{n\in \omega}\cA_n, B\in \cB\}$ is a witness for the required inequality.  

    Second, we prove that $\dnumber\leq \cofs(\conv_\alpha)$. 
    Let $\mathcal{A}$ be the witness for $\cofs(\conv_\alpha)$. For any $g\in \omega^\omega$ we define the set $B_g=\bigcup_{n<\omega}[\omega\cdot n, \omega \cdot n + g(n)]$ and for any $A\in\conv_\alpha$ we define $f_A=\{(n, \max\{k: \omega\cdot n + k \in A\}): n \in \omega\}$, with the convention that $\max\emptyset=0$ and $\max(C)=0$ for any infinite set $C$.  We will show $\mathcal{F}=\{f_A: A\in \mathcal{A}\}$ to be the dominating family in $\omega^\omega$. Take any $g\in\omega^\omega$. Since the  derivative of the set $B_g$ is contained in $\{\omega^2\}$, it follows that $B\in\conv_\alpha$, so there exists some set $A\in \mathcal{A}$ such that $B_g\subset^* A$, then $g(n)\leq f_A(n)$ for all but finitely many $n$. Hence  $\dnumber\leq |\mathcal{F}|=|\mathcal{A}|=\cofs(\conv_\alpha)$.
\end{proof}

%%%%%%%%%%%%%%%%%%%%%%%%%%%%%%%%%%%%%%%%%%%%%%%%%%
%%%%% REFERENCES
%%%%%%%%%%%%%%%%%%%%%%%%%%%%%%%%%%%%%%%%%%%%%%%%%%

\bibliographystyle{amsplain}
\bibliography{references}

\providecommand{\bysame}{\leavevmode\hbox to3em{\hrulefill}\thinspace}
\providecommand{\MR}{\relax\ifhmode\unskip\space\fi MR }
\providecommand{\MRhref}[2]{%
  \href{http://www.ams.org/mathscinet-getitem?mr=#1}{#2}
}
\providecommand{\href}[2]{#2}
\begin{thebibliography}{10}

\bibitem{MR3034318}
Pawe\l{} Barbarski, Rafa\l{} Filip\'{o}w, Nikodem Mro\.{z}ek, and Piotr Szuca,
  \emph{When does the {K}at\v{e}tov order imply that one ideal extends the
  other?}, Colloq. Math. \textbf{130} (2013), no.~1, 91--102. \MR{3034318}

\bibitem{MR1350295}
T.~Bartoszy\'{n}ski and H.~Judah, \emph{Set theory}, A K Peters, Ltd.,
  Wellesley, MA, 1995, On the structure of the real line. \MR{1350295}

\bibitem{MR2768685}
Andreas Blass, \emph{Combinatorial cardinal characteristics of the continuum},
  Handbook of set theory. {V}ols. 1, 2, 3, Springer, Dordrecht, 2010,
  pp.~395--489. \MR{2768685}

\bibitem{MR2899832}
Rafa\l{} Filip\'{o}w and Piotr Szuca, \emph{Three kinds of convergence and the
  associated {$\mathcal{I}$}-{B}aire classes}, J. Math. Anal. Appl.
  \textbf{391} (2012), no.~1, 1--9. \MR{2899832}

\bibitem{MR1181163}
J.~A. Fridy, \emph{Statistical limit points}, Proc. Amer. Math. Soc.
  \textbf{118} (1993), no.~4, 1187--1192. \MR{1181163}

\bibitem{MR3513296}
Osvaldo Guzm\'{a}n-Gonz\'{a}lez and David Meza-Alc\'{a}ntara, \emph{Some
  structural aspects of the {K}at\v{e}tov order on {B}orel ideals}, Order
  \textbf{33} (2016), no.~2, 189--194. \MR{3513296}

\bibitem{MR2319159}
F.~Hern\'{a}ndez-Hern\'{a}ndez and M.~Hru\v{s}\'{a}k, \emph{Cardinal invariants
  of analytic {$P$}-ideals}, Canad. J. Math. \textbf{59} (2007), no.~3,
  575--595. \MR{2319159}

\bibitem{MR3692233}
M.~Hru\v{s}\'{a}k, D.~Meza-Alc\'{a}ntara, E.~Th\"{u}mmel, and C.~Uzc\'{a}tegui,
  \emph{Ramsey type properties of ideals}, Ann. Pure Appl. Logic \textbf{168}
  (2017), no.~11, 2022--2049. \MR{3692233}

\bibitem{MR2777744}
Michael Hru\v{s}\'{a}k, \emph{Combinatorics of filters and ideals}, Set theory
  and its applications, Contemp. Math., vol. 533, Amer. Math. Soc., Providence,
  RI, 2011, pp.~29--69. \MR{2777744}

\bibitem{MR597342}
Kenneth Kunen, \emph{Set theory}, Studies in Logic and the Foundations of
  Mathematics, vol. 102, North-Holland Publishing Co., Amsterdam-New York,
  1980, An introduction to independence proofs. \MR{597342}

\bibitem{MR4584767}
A.~Kwela, \emph{Unboring ideals}, Fund. Math. \textbf{261} (2023), no.~3,
  235--272. \MR{4584767}

\bibitem{MR3594409}
A.~Kwela and J.~Tryba, \emph{Homogeneous ideals on countable sets}, Acta Math.
  Hungar. \textbf{151} (2017), no.~1, 139--161. \MR{3594409}

\bibitem{Mazurkiewicz-Sierpinski}
S.~Mazurkiewicz and W.~Sierpi\'{n}ski, \emph{Contribution \`{a} la topologie
  des ensembles d\'{e}nombrables}, Fund. Math. \textbf{1} (1920), no.~1,
  17--27.

\bibitem{alcantara-phd-thesis}
D.~Meza-Alc\'{a}ntara, \emph{Ideals and filters on countable set}, Ph.D.
  thesis, Universidad Nacional Aut\'{o}noma de M\'{e}xico, 2009,
  (\url{https://ru.dgb.unam.mx/handle/DGB_UNAM/TES01000645364}).

\bibitem{MR3550610}
Nikodem Mro\.{z}ek, \emph{Some applications of the {K}at\v{e}tov order on
  {B}orel ideals}, Bull. Pol. Acad. Sci. Math. \textbf{64} (2016), no.~1,
  21--28. \MR{3550610}

\bibitem{MR296671}
Zbigniew Semadeni, \emph{Banach spaces of continuous functions. {V}ol. {I}},
  Monografie Matematyczne [Mathematical Monographs], vol. Tom 55, PWN---Polish
  Scientific Publishers, Warsaw, 1971. \MR{296671}

\end{thebibliography}

\end{document}